\documentclass[10pt,reqno]{amsart}
\usepackage{graphicx,epsfig}
\usepackage{epstopdf}
\usepackage{pdfpages}
\usepackage{amssymb}
\usepackage{empheq}
\usepackage{cases}
\usepackage{amsthm,amsmath}
\usepackage{caption,lipsum}
\usepackage{stmaryrd}
\usepackage{tabularx}
\usepackage{color}
\usepackage{empheq,float}
\DeclareGraphicsExtensions{.eps}
\usepackage{hyperref}

\newtheorem{theorem}{Theorem}[section]
\newtheorem{lemma}[theorem]{Lemma}
\newtheorem{proposition}[theorem]{Proposition}
\newtheorem{corollary}[theorem]{Corollary}
\theoremstyle{definition}

\theoremstyle{remark}
\newtheorem{remark}[theorem]{Remark}

\newcommand{\fe}{\mathrm{e}}

\newcommand{\eps}{\varepsilon}
\newcommand{\bR}{{\mathbb R}}

\newcommand{\bT}{{\mathbb T}}
\newcommand{\bN}{{\mathbb N}}

\newcommand{\abs}[1]{\left\vert#1\right\vert}

\newcommand{\norm}[1]{\left\Vert#1\right\Vert}
\numberwithin{equation}{section}

\begin{document}

\title[Geometric two-scale integrators]{Geometric two-scale integrators for highly oscillatory system: uniform accuracy and near conservations}

\author[B. Wang]{Bin Wang}\address{\hspace*{-12pt}B.~Wang: School of Mathematics and Statistics, Xi'an Jiaotong University, 710049 Xi'an, China}
\email{wangbinmaths@xjtu.edu.cn}\urladdr{http://gr.xjtu.edu.cn/web/wangbinmaths/home}

\author[X. Zhao]{Xiaofei Zhao}
\address{\hspace*{-12pt}X.~Zhao: School of Mathematics and Statistics \& Hubei Key Laboratory of Computational Science, Wuhan University, Wuhan, 430072, China}
\email{matzhxf@whu.edu.cn}
\urladdr{http://jszy.whu.edu.cn/zhaoxiaofei/en/index.htm}

%

\date{}

\dedicatory{}

\begin{abstract}
In this paper, we consider a class of highly oscillatory Hamiltonian systems which involve a scaling parameter $\eps\in(0,1]$. The problem arises from many physical models in some limit parameter regime or from some time-compressed perturbation problems. The solution of the model exhibits rapid temporal oscillations with $\mathcal{O}(1)$-amplitude and $\mathcal{O}(1/\eps)$-frequency, which makes classical numerical methods inefficient. We apply the two-scale formulation approach to the problem and propose two new time-symmetric numerical integrators. The methods are proved to have the uniform second order accuracy for all $\eps$ at finite times and some near-conservation laws in long times. Numerical experiments on a H\'{e}non-Heiles model, a nonlinear Schr\"{o}dinger equation and a charged-particle system illustrate the performance of the proposed methods over the existing ones.
\end{abstract}

\maketitle

\section{Introduction}
This work concerns  a highly oscillatory canonical Hamiltonian system:
\begin{subequations}\label{model}
  \begin{align}
  &\dot{u}(t)=J\nabla H(u)=J\left[\frac{1}{\eps}Mu(t)+\nabla H_1(u(t))\right],\quad t>0,\\
  &u(0)=u_0,
  \end{align}
\end{subequations}
where $\eps\in(0,1]$ is a given parameter, $u(t):[0,\infty)\to X$ is the unknown living in a finite dimensional space or some functional space $X$ with $u_0$ the given initial data, $H_1(\cdot)$ is a given function, $M$ is a self-adjoint operator on $X$ and $J$ is the symplectic operator of $X$.
Here $J$ and $M$ are independent of time.
Along the equation (\ref{model}), the \emph{Hamiltonian} or \emph{energy} $H(t)$ is conserved:
\begin{equation}\label{energy}H(t)=H(u(t)):={\frac{1}{2\eps}}\langle Mu(t),u(t)\rangle+H_1(u(t))\equiv H(0),\quad t\geq0,\end{equation}
where $\langle\cdot,\cdot\rangle$ denotes the inner product on $X$.

The model problem (\ref{model}) arises from many different physical contexts, and the parameter $\eps$ is often introduced as some scaled physical quantity. In some limit physical regime, $\eps$ could be very small, i.e., $0<\eps\ll1$, and the solution of (\ref{model}) then becomes highly oscillatory in time. A precise example in the classical Newtonian dynamics could be the Fermi-Pasta-Ulam-Tsingou problem or the molecular dynamics \cite{Cohen thesis,Lubich}, where (\ref{model}) denotes a finite dimensional dynamical system. In such case, $J$ is the symplectic matrix, $M$ is a positive semi-definite diagonal matrix and $\eps$ is inversely proportional to the stiffness of the springs or the frequency of the atoms. For such so-called highly oscillatory differential equations, under the energy bounded condition, i.e., $H(t)$ being uniformly bounded as $\eps\to0$, extensive works have been done to understand the asymptotics of the equation (\ref{model}) and to design numerical methods with large time step and good long-time behaviour in the literature \cite{Cohen thesis,Cohen1,MD,Grimm,Hairer,Hochbruck1,WIW,Wu}. Note that the energy bounded condition $H(t)=\mathcal{O}(1)$ as $\eps\to0$ in (\ref{energy}) implies some small initial data assumption on $u_0$ (at least for some components). The large initial data case \cite{Zhao} would induce wilder oscillations in the solution $u(t)$ and consequently requests extra efforts in the design of numerical methods and analysis.
On the other hand, (\ref{model}) can also represent some multiscale partial differential equations. This occurs usually in quantum physics models. For instance, the nonlinear Klein-Gordon equation in the non-relativistic limit regime \cite{Masmoudi}, where $\sqrt{\eps}$ denotes the inverse of the speed of light,  can be formulated into the form of (\ref{model}) by some change of variable \cite{PI1,KGC,UAKG}. In such case, the initial data $u_0$ is considered as a $\mathcal{O}(1)$ function in the space variable and the energy (\ref{energy}) is unbounded as $\eps\to0$ \cite{Masmoudi}.

Another class of problem that is closely related to (\ref{model}) is the long-term dynamics of a perturbed Hamiltonian system \cite{Faou,Gauckler,Lubich,Kuksin}:
\begin{equation}\label{perturb}
\dot{\mathfrak{u}}\left(\mathfrak{t}\right)=J\left[M\mathfrak{u}\left(\mathfrak{t}\right)
+\eps\nabla H_1\left(\mathfrak{u}\left(\mathfrak{t}\right)\right)\right],\quad \mathfrak{t}>0,\quad
\mathfrak{u}(0)=u_0.\end{equation}
It has been a classical problem for many years.
When $\nabla H_1(\mathfrak{u})$ is a higher order nonlinearity of $\mathfrak{u}$,
 (\ref{perturb}) is equivalent to a non-perturbation problem but with small initial data. Here we consider the same $u_0=\mathcal{O}(1)$ for (\ref{perturb}) as in (\ref{model}).
Then, it is clear that  (\ref{model}) and (\ref{perturb}) is connected by the relation: \begin{equation}\label{connect}
\mathfrak{t}=t/\eps,\quad u(t)=\mathfrak{u}(\mathfrak{t}).
\end{equation}
 That is to say (\ref{model}) can be considered as the compression of (\ref{perturb}) from a large time interval, and {hence} the solution $u(t)$ of (\ref{model}) is highly oscillatory in the time scale of $t$ for small $\eps$. For the perturbed problem (\ref{perturb}), deep theories such as the KAM theory \cite{Lubich,Kuksin}, and powerful tools such as the normal form transform \cite{Lubich,Kuksin} and the modulated Fourier expansion \cite{Cohen1,Cohen2,Lubich ICM,Lubich} have been developed for the long-time analysis. Numerically to integrate (\ref{perturb}) for a long time, trigonometric/exponential integrators \cite{Cohen thesis,Cohen0,Lubich,WW} and time-splitting methods \cite{Feng, Faou,Faou1,Faou2,Gauckler,Gauckler thesis,Lubich2} have been considered and analyzed. Though these methods under certain conditions could offer nice near-conservation laws for long times, the use of a time step $\Delta \mathfrak{t}=\mathcal{O}(1)$ for integrating (\ref{perturb}) over a large time interval, e.g., a size of $\mathcal{O}(1/\eps)$, makes the practical computations inefficient.

In this work, we focus on numerically solving the highly oscillatory scaling problem (\ref{model}). Though mathematically (\ref{model}) and (\ref{perturb}) are equivalent, numerically they are not. On the one hand, by solving (\ref{model}) to a long time numerically, our computation could indeed get us to a `longer' time than that for (\ref{perturb}). On the other hand, we are interested in solving (\ref{model}) accurately and efficiently with a time step $\Delta t=\mathcal{O}(1)$ for all the $\eps\in(0,1]$.
Note that $\Delta t=\mathcal{O}(1)$ for integrating (\ref{model}) corresponds to a time step $\Delta \mathfrak{t}=\mathcal{O}(1/\eps)$ for (\ref{perturb}). This sounds too good to be true for standard numerical discretizations on (\ref{model}) due to the unbounded temporal derivatives of the solution $u(t)$ as $\eps\to0$. The  methods usually require $\Delta t$ being bounded by $\eps$ to some extent for the convergence and stability. While, such property can be obtained on (\ref{model})
by the so-called \emph{uniformly accurate (UA)} methods, whose accuracy and computational costs are independent of $\eps\in(0,1]$.

In the past few years, UA methods have been extensively constructed via different approaches. For example, it can be constructed by some multiscale expansion of the solution which decomposes the original equation into several parts to solve \cite{Baocai,BCZ,Baosu,Cohen thesis,NUA,Zhao}. Also, it can be constructed by some filtered  variable followed by the Picard-type iterations \cite{PI1,PI2}. Moreover, the averaging theory can give rise to the UA schemes as well \cite{UAKG,Jin,2scale multi,APVP2d,NUA,autoUA}. The UA property in these approaches can be established for (\ref{model}) till a fixed finite time (independent of $\eps$). While, a pity within most of the UA methods so far (one exception: the pullback method \cite{NUA} with numerical evidence but yet to analyze) is the lack of good behaviour for solving (\ref{model}) to a long time.  The numerical errors of the invariants such as the energy (\ref{energy}) increase quickly as time evolves \cite{KGC,vp3D,NUA,UAVP4d}.
The reason is due to the loss of the geometric structure of the equation under the aforementioned approaches. For instance, the mutiscale expansion approach involves a remainder's equation that is non-autonomous \cite{BCZ,NUA}, and the direct Picard-type iteration breaks the geometry \cite{PI1}. Thus, it remains a challenge to get a geometric UA scheme.

The aim of the work is to propose and analyze some UA schemes with good long-time conservation laws for solving (\ref{model}) for $\eps\in(0,1]$. To get the UA property, we consider the \emph{two-scale formulation} approach presented in \cite{UAKG,APVP2d}. The approach works by separating the principle fast time scale out in the solution as an  extra variable, and then solves numerically the augmented equation (called as the two-scale equation). So far, some finite difference schemes \cite{UAKG,APVP2d} and exponential integrators \cite{vp3D,UAVP4d} have been proposed to discretize the two-scale equation in time for achieving the UA property. However, numerical blow-ups or energy drifts have been observed for the long-time computing in applications \cite{KGC,vp3D,UAVP4d}. It is unclear whether this is due to the extended dimension in the two-scale equation and/or its discretization. To address this concern, we propose some symmetric integrators for solving the two-scale equation, which we call the \emph{symmetric exponential-type two-scale
integrators (SE-TSIs)}. We prove the UA property of SE-TSIs as the solver for (\ref{model}) at finite times. Then, we apply the proposed methods to a H\'{e}non-Heiles model \cite{Henon}, a nonlinear Schr\"odinger equation (NLS) \cite{SAV} and a charged-particle dynamics (CPD) \cite{Frenod} in the form of (\ref{model}), where our numerical results not only verify the second order UA convergence rate, but also illustrate the near conservations in long times for some widely concerned quantities, such as the mass and the energy.
Analysis of the proposed SE-TSIs in long times are done in the end by the technique of the modulated Fourier expansions which reveals the near-conservation laws rigorously.

The rest of the paper is organized as follows. In Section \ref{sec:2review}, we briefly review the framework of the two-scale formulation, and then we present the SE-TSI schemes as well as  their finite-time convergence results. Applications to the H\'{e}non-Heiles, NLS and CPD models are made in Section \ref{sec:4num} with numerical results presented to show the UA property and the near conservations. In Section \ref{sec:5LTA}, we perform the long-time analysis for the SE-TSIs, and some conclusions are drawn in Section \ref{sec:6con}.

\section{Numerical method and uniform accuracy}\label{sec:2review}
In this section, we shall firstly review the two-scale formulation strategy for the highly oscillatory problem (\ref{model}), as well as some existing uniformly accurate numerical solvers based on the strategy. Then, we shall present the symmetric two-scale integrators and their finite-time convergence results.

\subsection{Two-scale formulation in a nutshell}\label{sec2.1}
In most of the application cases, the operator $M$ in (\ref{model}) is diagonal, and it commutes with $J$, i.e., $JM=MJ$. Moreover, for the simplicity of the presentation, we assume in the following that the propagator $\tau\to\fe^{JM\tau}$ generated by the stiff part in (\ref{model}) is $2\pi$-periodic. If it is not directly the case, one may decompose a leading part from $M$, i.e.,
$M=M_0+O(\eps)$ by some asymptotical expansion \cite{UAKG} or the Diophantine approximation \cite{2scale multi}, where the resulting $M_0$ has all the eigenvalues being integer-multiples of a common frequency (mono-frequency).

 By a change of variable $v(t)=\fe^{-JMt/\eps}u(t)$, the equation (\ref{model}) becomes:
\begin{equation}\label{filtered}
\dot{v}(t)=\fe^{-JMt/\eps}J\nabla H_1\left(\fe^{JMt/\eps}v(t)\right),\quad t>0,\quad v(0)=u_0.
\end{equation}
For simplicity, we denote
\begin{equation}\label{ftau def}
f_\tau(v):=\fe^{-JM\tau}J\nabla H_1\left(\fe^{JM\tau}v\right).
\end{equation}
  Under the assumption, $f_\tau(v)$ is a periodic function in $\tau$ with period $2\pi$.
The two-scale formulation works by separating the fast time variable $t/\eps$ in (\ref{filtered}) out from the slow variable $t$ and assuming it as another independent variable $\tau$. Then, the two-scale equation of (\ref{filtered}) or (\ref{model}) reads \cite{UAKG,APVP2d}
\begin{numcases}
  \,\partial_tU(t,\tau)+\frac{1}{\eps}\partial_\tau U(t,\tau)=f_\tau(U(t,\tau)),\quad t>0,\ \tau\in\bT,\label{2scale eq}\\
  U(0,0)=u_0,\nonumber
\end{numcases}
where $\bT=(0,2\pi)$ is a torus, and $U(t,\tau)$ is the unknown satisfying the periodic boundary condition in $\tau$.
Under the condition $U(0,0)=u_0$ at the initial time, by the uniqueness of the solution of (\ref{model}), the two-scale solution of (\ref{2scale eq}) gives the original solution of (\ref{model}) by $$U(t,t/\eps)=v(t)=\fe^{-JMt/\eps}u(t),\quad t\geq0.$$

The variable $\tau\in\bT$ offers a degree of freedom for choosing the initial data $U(0,\tau)$ for the two-scale equation (\ref{2scale eq}). By constructing the suitable $U(0,\tau)$, one is able to make $U(t,\tau)$ as smooth as possible in terms of $\eps$, i.e., for any desired $k\in\bN$,
$
\partial_t^jU(t,\tau)=\mathcal{O}(1),\  \eps\to0,\  j=0,1,\ldots,k.
$
 Such so-called \emph{$k$-th order well-prepared initial data} can be obtained in two means. As originally proposed in \cite{UAKG}, by treating the stiff part $1/\eps\partial_\tau U$ as a kind of collision term in the kinetic theory, one can perform a Chapman-Enskog expansion for the two-scale solution.  For the later use, let us briefly describe such process in the following for the initial data up to the second order.

 Denote $L=\partial_\tau$ and the averaging operator $\Pi w:=\frac{1}{2\pi}\int_0^{2\pi}w(s)ds$ for some $w(\cdot)$ on $\bT$. The Chapman-Enskog expansion reads: $U(t,\tau)=\underline{U}(t)+r(t,\tau)$ with $\underline{U}=\Pi U$. It decomposes (\ref{2scale eq}) into
 $$\dot{\underline{U}}(t)=\Pi f_\tau(U(t,\tau)),\quad \partial_t r+\frac{1}{\eps}
 Lr=(I-\Pi)f_\tau(U),$$
 which implies $r=\eps \mathcal{A}f_\tau(U)-\eps L^{-1}\partial_t r$ and $\partial_tr=\eps \mathcal{A}\nabla f_\tau(U)(\dot{\underline{U}}+\partial_tr)-\eps L^{-1}\partial_t^2r$ with $\mathcal{A}:=L^{-1}(I-\Pi)$.
 By assuming $\partial_t^2r=O(1)$ which is desired, one finds $r,\,\partial_tr=O(\eps)$ and so
 $r=\eps \mathcal{A}f_\tau(\underline{U})+O(\eps^2)$. Note that by the condition
 $u_0=U(0,0)=\underline{U}(0)+r(0,0)$, one has
 $U(0,\tau)=u_0+r(0,\tau)-r(0,0)$ and
 $r(0,\tau)=\eps \mathcal{A} f_\tau(u_0)+O(\eps^2)$. Thus, the second order initial data defined as
 \begin{equation}\label{2nd data}
   U(0,\tau)=u_0+r_{2nd}(0,\tau)-r_{2nd}(0,0),\quad r_{2nd}(0,\tau):=\eps \mathcal{A} f_\tau(u_0),
 \end{equation}
 gives the correct leading order terms for the solution at $t=0$ when $\partial_t^2U=O(1)$.
 Indeed, by taking (\ref{2nd data}) for (\ref{2scale eq}), the desired key property
holds as following (see more details in \cite{UAKG}). One can simply think of the finite dimensional case of $X$, but it is also true for the functional space case.
\begin{proposition}\emph{(\cite{UAKG,autoUA})} \label{prop1}
 Let $u_0\in K$ with $K$ a bounded open subset of $X$, and assume that the solution of \eqref{model} stays in $K$ for $t\in[0,T_0]$ with some $T_0>0$. Let $H_1(u)$ be smooth  and uniformly bounded on the closure of $K$ for all $\eps$. Then for \eqref{2scale eq} associated with the initial data \eqref{2nd data}, there exist some constant $\eps_0>0$ and some constant $C>0$  independent of $\eps$ such that for $\eps\in (0,\eps_0]$,
\begin{equation}\label{bound}
\sup_{t\in(0,T_0)}\left\{\left\|\partial_t^{k}U(t,\cdot)\right\|_{H^1_\tau(X)}\right\}\leq C,\quad k=0,1,2.
\end{equation}
\end{proposition}
\begin{proof}
 By (\ref{ftau def}), the smoothness of $H_1(\cdot)$ guarantees the boundedness of $f_\tau(\cdot)$ and its derivative with respect to $\tau$. The rest of the proof follows \cite{UAKG} by taking the $k$-th order time derivative of (\ref{2scale eq}) and using the bootstrap type argument.
\end{proof}

The uniform boundedness of the derivatives (\ref{bound}) give rise to possible numerical approximations of (\ref{2scale eq}) that are uniformly accurate (UA) in $\eps$.
 Alternatively, $U(0,\tau)$ can be constructed based on the averaging theory \cite{SAV,Sanders}. 
{In such a way}, one can get the well-prepared initial data to arbitrary order in a recursive manner and we refer the readers to \cite{autoUA} for more details.


Now we review some existing two-scale numerical integrators that were proposed to get the second order UA property. These methods will be the benchmark for our numerical illustrations later.
Let $h=\Delta t>0$ be the time step with $t_n=nh$ for $n\geq0$, and denote the numerical solution for (\ref{2scale eq}) as $U^n=U^n(\tau)\approx U(t_n,\tau)$ with $U^0=U(0,\tau)$. A finite difference two-scale integrator (FD-TSI) was proposed in \cite{UAKG}:
\begin{equation}\label{FD scheme}
  \left\{\begin{split}
  &\frac{U^{n+\frac12}-U^n}{h/2}+\frac{1}{\eps}\partial_\tau U^{n+\frac12}=f_\tau\left(U^{n}\right),\\
  &\frac{U^{n+1}-U^n}{h}+\frac{1}{2\eps}\partial_\tau\left(U^{n+1}
  +U^n\right)=f_\tau\left(U^{n+\frac12}\right),\quad n\geq0.
  \end{split}\right.
\end{equation}
It can be checked directly that the leading part of the local truncation error depends on
$\partial_{t}^3U$. Therefore, by taking the third order initial data \cite{UAKG}, one can get the uniform second order accuracy from (\ref{FD scheme}) for solving (\ref{model}). To reduce the required order of the prepared initial data, a multistep exponential-type two-scale integrator (ME-TSI) was proposed in \cite{UAVP4d}:
\begin{numcases}
\,U^{n+1}=\fe^{-\frac{h\partial_\tau}{\eps}}U^n+h\varphi_1(-h\partial_\tau/\eps)f_\tau\left(U^n\right)+
  h\varphi_2(-h\partial_\tau/\eps)\left[f_\tau\left(U^n\right)-f_\tau\left(U^{n-1}\right)\right],\quad  n\geq1,\nonumber \\
  U^{1}=\fe^{-\frac{h\partial_\tau}{\eps}}U^0+h\varphi_1(-h\partial_\tau/\eps)f_\tau\left(U^0\right)+
  h\varphi_2(-h\partial_\tau/\eps)\left[f_\tau\left(U^*\right)-f_\tau\left(U^{0}\right)\right], \label{EI scheme}\\ U^*=\fe^{-\frac{h\partial_\tau}{\eps}}U^0+h\varphi_1(-h\partial_\tau/\eps)f_\tau\left(U^0\right),\nonumber
  \end{numcases}
with the {notation} $\varphi_k(z)=\int_0^1
\theta^{k-1}\fe^{(1-\theta)z}/(k-1)!d\theta$ for $k=1,2$ from \cite{Ostermann}. It can be seen that the local truncation error of ME-TSI depends on $\partial_{t}^2U$, and so the second order data (\ref{2nd data}) is enough for the uniform second order accuracy. For the $\tau$-variable in the above two schemes, thanks to the periodic boundary condition in (\ref{2scale eq}),  the Fourier pseudospectral method \cite{Shen,Trefethen} can be further applied  to get the full discretizations, which works efficiently in practice via the fast Fourier transform.

The above two numerical methods are UA, but in practical applications they only work till a short time. Numerical blow-ups or rapid energy drifts have been observed in the long-time computations \cite{KGC,vp3D,autoUA,UAVP4d}. Although a practical restart strategy (stop (\ref{2scale eq}) at a fixed time period, re-prepare the initial data and restart the method) has been proposed in \cite{vp3D,UAVP4d} to enhance the performance, there is still a significant energy drift and it was not understood  whether this is due to the two-scale formulation or the numerical method.
Here we note that the two-scale equation (\ref{2scale eq}) or the constructed data
(\ref{2nd data}) only enlarges the dimension but breaks nothing, and so the system should maintain the original geometry of (\ref{model}) in a lower dimension. Such belief motivates us to consider some numerical integrators on (\ref{2scale eq}) that could be both UA and geometric.

\subsection{Symmetric two-scale integrators}\label{sec:3S-TSI}
Based on the Duhamel formula for (\ref{2scale eq}):
\begin{equation}\label{Duhamel}
U(t_n+s)=\fe^{-\frac{s\partial_\tau}{\eps}}U(t_n)+\int_0^s\fe^{(\theta-s)\frac{\partial_\tau}{\eps}}
f_\tau(U(t_n+\theta))d\theta,\quad 0\leq s\leq h,\quad n\geq0,
\end{equation}
we propose here two \emph{symmetric exponential-type two-scale integrators (SE-TSIs)} for solving the equation (\ref{2scale eq}).
The first one is an one-stage type scheme (referred as \textbf{SE1-TSI}):
\begin{subequations}\label{S1}
\begin{numcases}%
\,U^{n+\frac{1}{2}}=\fe^{-\frac{h\partial_\tau}{2\eps}}U^n+\frac{1}{2}h\varphi_1 \left(-h\partial_\tau/(2\eps)\right)
f_\tau\left(U^{n+\frac{1}{2}}\right),\label{S1 a}\\
U^{n+1}=\fe^{-\frac{h\partial_\tau}{\eps}}U^n+h
\varphi_1 (-h\partial_\tau/\eps)f_\tau\left(U^{n+\frac{1}{2}}\right),\quad n\geq0.\label{S1 b}
\end{numcases}
\end{subequations}
The second one is an average-vector-field type {\cite{Li,AVF}} scheme (referred as \textbf{SE2-TSI}):
\begin{equation}\label{S2}
U^{n+1}=\fe^{-\frac{h\partial_\tau}{\eps}}U^n+h
\varphi_1 (-h\partial_\tau/\eps)\int_{0}^1
f_\tau\left((1-\rho) U^{n}+\rho U^{n+1}\right) d\rho.
\end{equation}

The existence of the numerical solution in the nonlinear equation of (\ref{S1}) or (\ref{S2}) is clearly guaranteed by the implicit function theorem.
The full discretizations of the two can be obtained again by the Fourier pseudospectral method.
Moreover, we have the time-symmetric property for the two SE-TSIs stated as follows.
\begin{proposition}\emph{(Time symmetry).}
  The schemes SE1-TSI \eqref{S1} and SE2-TSI \eqref{S2} are symmetric in time, i.e., by exchanging $n+1\leftrightarrow n$ and $h\leftrightarrow -h$, \eqref{S1} and \eqref{S2} remain the same.
\end{proposition}
\begin{proof}
Note that $\varphi_1(z)=(\fe^{z}-1)/z$ and $\fe^{-z}\varphi_1(z)=\varphi_1(-z)$, and then it is direct to see the symmetry in (\ref{S2}) for SE2-TSI and the symmetry in (\ref{S1 b}) for SE1-TSI. As for (\ref{S1 a}), under $n+1\leftrightarrow n$ and $h\leftrightarrow -h$, it becomes
$U^{n+\frac{1}{2}}=\fe^{\frac{h\partial_\tau}{2\eps}}U^{n+1}-\frac{1}{2}h\varphi_1 \left(h\partial_\tau/(2\eps)\right)f_\tau(U^{n+\frac{1}{2}})$. Since (\ref{S1 b}) is unchanged and note $\fe^{z/2}\varphi_1(-z)-\frac12\varphi_1(z/2)=\frac12\varphi_1(-z/2)$, so we find
\begin{align*}
 U^{n+\frac{1}{2}}&=\fe^{\frac{h\partial_\tau}{2\eps}}\left[\fe^{-\frac{h\partial_\tau}{\eps}}
 U^n+h\varphi_1(-h\partial_\tau/\eps)f_\tau\left(U^{n+\frac{1}{2}}\right)\right]
 -\frac{1}{2}h\varphi_1 \left(h\partial_\tau/(2\eps)\right)f_\tau\left(U^{n+\frac{1}{2}}\right)\\
 &=\fe^{-\frac{h\partial_\tau}{2\eps}}U^n+\frac{1}{2}h\varphi_1 \left(-h\partial_\tau/(2\eps)\right)
f_\tau\left(U^{n+\frac{1}{2}}\right),
\end{align*}
which verifies the assertion.
\end{proof}

For the finite-time convergence of the two SE-TSIs, we have the following result showing that
they offer the uniform second order accuracy as numerical methods for solving \eqref{model}, if one prepares the initial data \eqref{2nd data} for the two-scale equation \eqref{2scale eq}.

\begin{proposition}\label{prop2}\emph{(Uniform convergence).} Consider the initial data \eqref{2nd data} for \eqref{2scale eq} and let $U^n$ be obtained from SE1-TSI \eqref{S1} or SE2-TSI \eqref{S2}. Under the assumption of Proposition \ref{prop1}, there exist constants $h_0,C>0$ independent of $\eps$ such that for $0<h\leq h_0$,
 $$\|U^n(t_n/\eps)-u(t_n)\|_{X}\leq Ch^2,\quad 0\leq n\leq T_0/h.$$
\end{proposition}
\begin{proof}
Let us omit the $\tau$-variable for simplicity, e.g., $U(t)=U(t,\tau)$. Firstly for SE1-TSI,
its local truncation error by (\ref{Duhamel}) consists of the error from the right rectangle rule at the half-step:
\begin{align}
 \xi_1^n:=&U\left(t_{n+\frac{1}{2}}\right)-\fe^{-\frac{h\partial_\tau}{2\eps}}U(t_n)-\frac{1}{2}h\varphi_1 \left(-h\partial_\tau/(2\eps)\right)
f_\tau\left(U\left(t_{n+\frac{1}{2}}\right)\right)\label{xi_p def}\\
=&\int_0^{h/2}\fe^{(\theta-\frac{h}{2})\frac{\partial_\tau}{\eps}}
\int_{h/2}^\theta\nabla f_\tau(U(t_{n}+\rho))\partial_tU(t_n+\rho)d\rho d\theta,\nonumber
\end{align}
 and the error from the midpoint rule at the integer-step:
  \begin{align*}
    &\xi_2^n:=U(t_{n+1})-\fe^{-\frac{h\partial_\tau}{\eps}}U(t_n)-
    h\varphi_1 (-h\partial_\tau/\eps)f_\tau\left(U\left(t_{n+\frac{1}{2}}\right)\right)\\
    =&\int_0^{\frac{h}{2}}\fe^{(\theta-\frac{h}{2})\frac{\partial_\tau}{\eps}}
    \left(\theta-\frac{h}{2}\right)^2\int_0^1\rho\left[\nabla^2f_\tau(U(t))(\partial_tU(t))^2
    +\nabla f_\tau(U(t))\partial_t^2U(t)\right]_{t=t_n+\rho\theta+(1-\rho)\frac{h}{2}}d\rho d\theta.
  \end{align*}
  Note that $\fe^{t\partial_\tau}$ for $t\in\bR$ is isometric in $H^1_\tau$. So
by the assumptions and the result of Proposition \ref{prop1}, we have $\|\xi_1^n\|_{H^1_\tau}\leq Ch^2$ and $\|\xi_2^n\|_{H^1_\tau}\leq Ch^3$. Here and after, $C>0$ denotes some constant independent of $\eps$ but whose value may vary line by line.
Define $U_*^{n+\frac{1}{2}}$ satisfying
\begin{equation}\label{Ustar def}
U_*^{n+\frac{1}{2}}-\fe^{-\frac{h\partial_\tau}{2\eps}}U(t_n)-\frac{1}{2}h\varphi_1 \left(-h\partial_\tau/(2\eps)\right)f_\tau\left(U_*^{n+\frac{1}{2}}\right)=0,
\end{equation}
which implies  $\|U_*^{n+\frac12}\|_{H^1_\tau}\leq\|U(t_n)\|_{H^1_\tau}+1$ for all $n$ when $h$ is smaller than a constant independent of $\eps$.
Then by (\ref{xi_p def}) and noting $\|\varphi_1(t\partial_\tau)w\|_{H^1_\tau}\leq\|w\|_{H^1_\tau}$ for any $w(\tau)$, it can be seen that
$$\left\|U\left(t_{n+\frac12}\right)-U_*^{n+\frac{1}{2}}\right\|_{H^1_\tau}\leq C\|\xi_1^n\|_{H^1_\tau}.$$
The local truncation error of SE1-TSI in total reads
  \begin{align}
    \xi^n:=&U\left(t_{n+1}\right)-\fe^{-\frac{h\partial_\tau}{\eps}}U(t_n)-
    h\varphi_1 (-h\partial_\tau/\eps)f_\tau\left(U_*^{n+\frac{1}{2}}\right)\label{xi s1 def}\\
    =&\xi_2^n+h\varphi_1 (-h\partial_\tau/\eps)\left[f_\tau\left(U\left(t_{n+\frac{1}{2}}\right)\right)-
    f_\tau\left(U_*^{n+\frac{1}{2}}\right)\right],\nonumber
  \end{align}
  and so $\|\xi^n\|_{H^1_\tau}\leq C(\|\xi_2^n\|_{H^1_\tau}+h\|\xi_1^n\|_{H^1_\tau})\leq Ch^3.$
   With $e_n:=U(t_n)-U^n$, (\ref{xi s1 def}) and (\ref{S1 b}) give
  \begin{align*}
   e_{n+1}=\fe^{-\frac{h\partial_\tau}{\eps}}e_n+h\varphi_1 (-h\partial_\tau/\eps)\left[f_\tau\left(U_*^{n+\frac{1}{2}}\right)-
    f_\tau\left(U^{n+\frac{1}{2}}\right)\right]+\xi^n,
  \end{align*}
  and from (\ref{Ustar def}) and (\ref{S1 a}),
  $$U_*^{n+\frac{1}{2}}-
    U^{n+\frac{1}{2}}=\fe^{-\frac{h\partial_\tau}{2\eps}}e_n+\frac{h}{2}\varphi_1 (-h\partial_\tau/(2\eps))\left[f_\tau\left(U_*^{n+\frac{1}{2}}\right)-
    f_\tau\left(U^{n+\frac{1}{2}}\right)\right].$$
    Then by taking the $H^1_\tau$-norm for the above two and the induction for the boundedness of $U^n$ and $U^{n+\frac12}$ (or the Lady
Windermere's fan argument \cite{Norsett}), we get $\|e_n\|_{H^1_\tau}\leq Ch^2$ for $0\leq n\leq \frac{T_0}{h}$ which implies
$\|e_n\|_{L^\infty_\tau}\leq Ch^2$. By the fact $U(t,t)=u(t)$, we conclude
$\|U^n(t_n/\eps)-u(t_n)\|_X\leq Ch^2$.

Next, we consider SE2-TSI (\ref{S2}). By the Taylor expansion, we find
\begin{align}\label{delta def}
 \delta^n=U(t_n+h\rho)-(1-\rho)U(t_n)-\rho U(t_{n+1})=h^2\rho(\rho-1)\int_0^1(1-s)\partial_t^2 U(t_n+hs)ds,
\end{align}
and so $\|\delta^n\|_{H^1_\tau}\leq C h^2$. By (\ref{Duhamel}), we note
$U(t_{n+1})=\fe^{-\frac{h\partial_\tau}{\eps}}U(t_n)+h\int_0^1\fe^{h(\rho-1)\frac{\partial_\tau}{\eps}}
f_\tau(U(t_n+h\rho))d\rho$, {and then} the local truncation error of SE2-TSI (\ref{S2}) reads
\begin{align}
\zeta^n:=&U(t_{n+1})-\fe^{-\frac{h\partial_\tau}{\eps}}U(t_n)-h
\varphi_1 (-h\partial_\tau/\eps)\int_{0}^1
f_\tau\big((1-\rho) U(t_{n})+\rho U(t_{n+1})\big) d\rho\label{zeta def}\\
=&h\int_0^1\fe^{h(\rho-1)\frac{\partial_\tau}{\eps}}\left[
f_\tau(U(t_n+h\rho))-
f_\tau\big((1-\rho) U(t_{n})+\rho U(t_{n+1})\big)\right] d\rho.\nonumber
\end{align}
Then (\ref{delta def}) tells
$\|\zeta^n\|_{H^1_\tau}\leq C h^3$. The difference between (\ref{zeta def}) and (\ref{S2}) gives the error equation:
\begin{align*}
 e_{n+1}=&\fe^{-\frac{h\partial_\tau}{\eps}}e_n+h
\varphi_1 (-h\partial_\tau/\eps)\int_{0}^1
\left[f_\tau\left((1-\rho) U(t_{n})+\rho U(t_{n+1})\right)-
f_\tau\left((1-\rho) U^{n}+\rho U^{n+1}\right)\right] d\rho\\
&+\zeta^n,
\end{align*}
 and the rest proceeds similarly as before.
\end{proof}

We end this section by remarking that the two-scale equation \eqref{2scale eq} may possess some
sympletic structure. Finding such intrinsic structure
and the corresponding structure-preserving numerical schemes would be of independent interests to us.
From the practical point of view, our proposed symmetric methods (\ref{S1}) and (\ref{S2}) already can achieve the good long-term performance for approximating (\ref{model}), which will be illustrated in the coming section by some numerical examples.


\section{Applications and numerical results}\label{sec:4num}
In this section, we apply the proposed two SE-TSIs to solve some precise examples in the form of (\ref{model}). We shall test the accuracy and the long-term performance of the methods on each example. All the methods are programmed here directly without involving the restart strategy from \cite{vp3D,UAVP4d}.

\subsection{H\'{e}non-Heiles model}
Our first numerical example is devoted to the
  H\'{e}non-Heiles model, which is a classical Hamiltonian system from astronomy  \cite{Henon,Lubich}. We take the form as in \cite{NUA,autoUA}:
  \begin{equation}\label{HH model}
\frac{\textmd{d}}{\textmd{d} t}\begin{pmatrix}
                                   q_1 \\
                                   q_2 \\
                                   p_1 \\
                                   p_2
                                 \end{pmatrix}
= \frac{1}{\eps}J{\left(
     \begin{array}{cccc}
        1& 0 & 0 & 0 \\
       0 &0 & 0 & 0 \\
       0 & 0 & 1 & 0 \\
       0 & 0 &0 & 0 \\
     \end{array}
   \right)}\begin{pmatrix}
                                   q_1 \\
                                   q_2 \\
                                   p_1 \\
                                   p_2
                                 \end{pmatrix}+J\nabla H_1(q_1,q_2,p_1,p_2),\quad t>0,
\end{equation}
 with $J=\binom{\ \ 0 \quad\ I_{2\times2}}{-I_{2\times2} \ \ \, 0}
 $
 and {$H_1=q_2^2/2+p_2^2/2+q_1^2q_2-q_2^3/3.$}
With {$u=(q_1,q_2,p_1,p_2)^\intercal$}, (\ref{HH model}) is exactly in the form of (\ref{model}) and meets the assumptions in Section \ref{sec2.1}.
We choose the initial value {$u_0=(0.12,0.12,0.12,0.12)^\intercal$} and apply the presented numerical schemes SE1-TSI (\ref{S1}) and SE2-TSI (\ref{S2}) to solve (\ref{HH model}). The reference solution is obtained by  the `ode45' of MATLAB. For the discretization in the $\tau$-direction of all the  methods,
we fix $N_\tau=2^6$ so that this part of error is rather negligible. The two proposed SE-TSIs are implicit, and so a fixed point iteration is applied with the error tolerance and the maximum number of each iteration set as  $10^{-10}$ and $200$, respectively.

\begin{figure}[t!]
$$\begin{array}{cc}
\psfig{figure=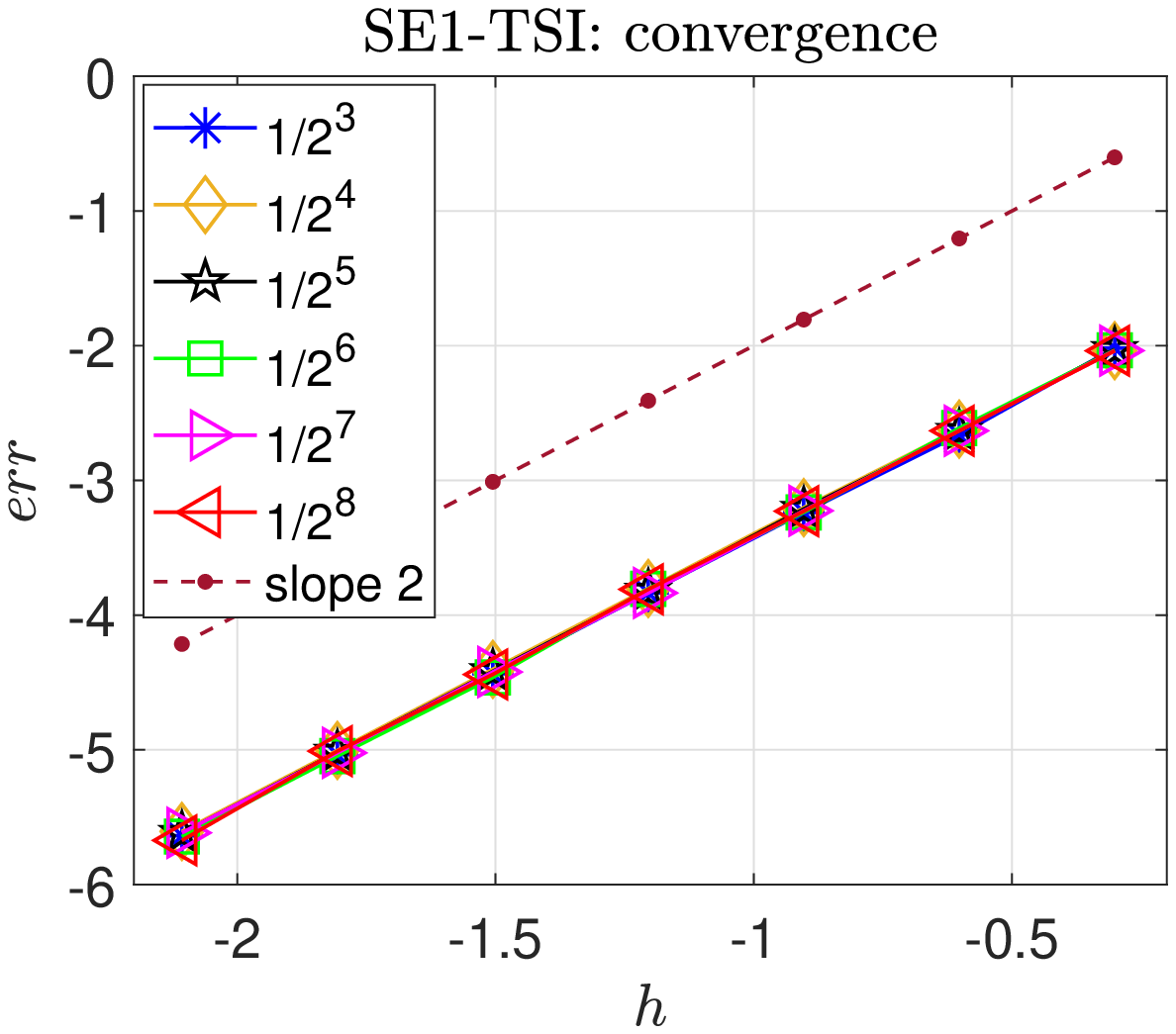,height=4.5cm,width=6.7cm}
\psfig{figure=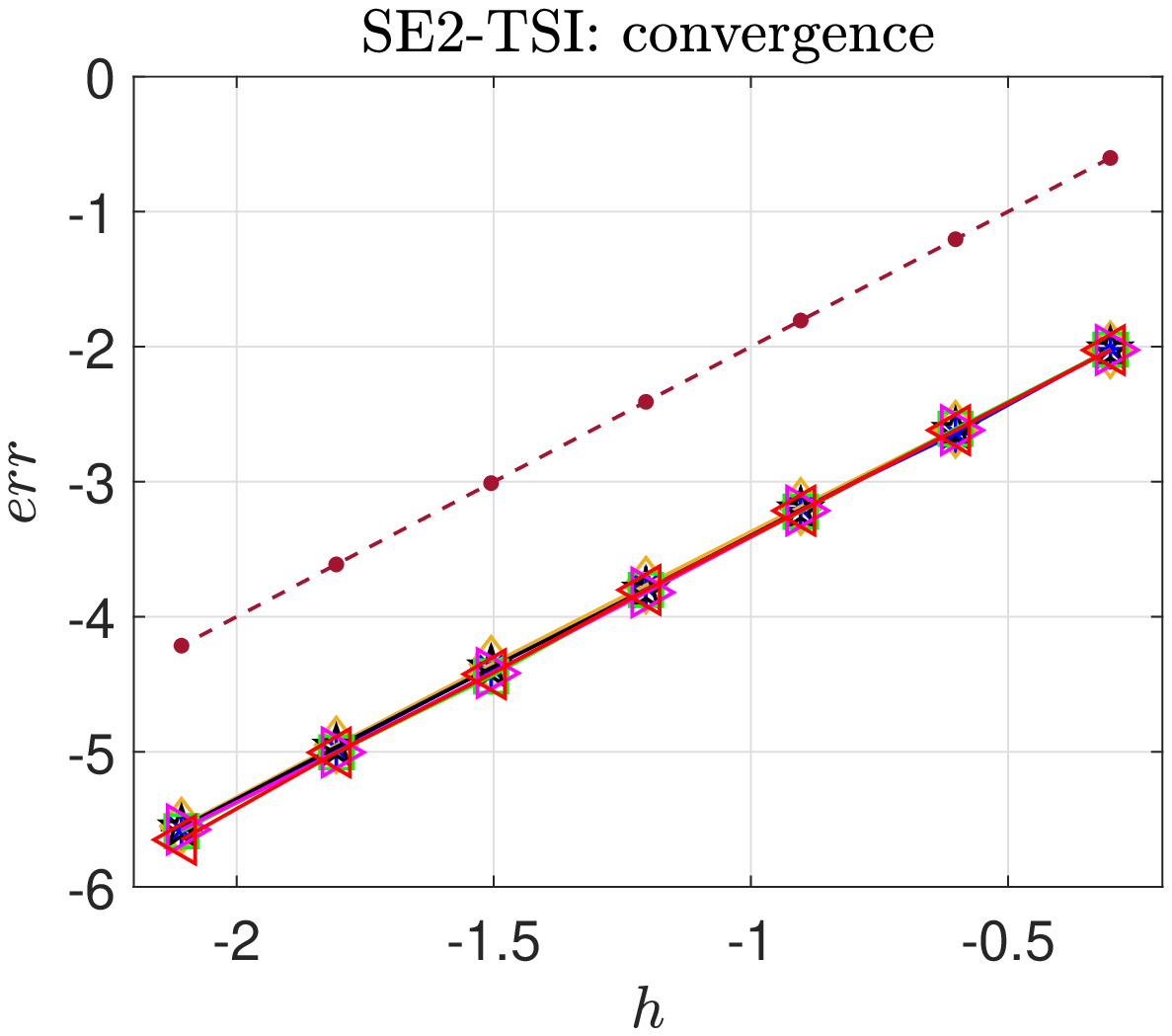,height=4.5cm,width=6.7cm}
\end{array}$$
\caption{H\'{e}non-Heiles example (\ref{HH model}): the log-log plot of the temporal error $err=\abs{u^n-u(t_n)}/\abs{u(t_n)}$ of SE-TSIs at $t_n=1$ under different $\eps$.}\label{fig10}
\end{figure}
   \begin{figure}[t!]
$$\begin{array}{cc}
\psfig{figure=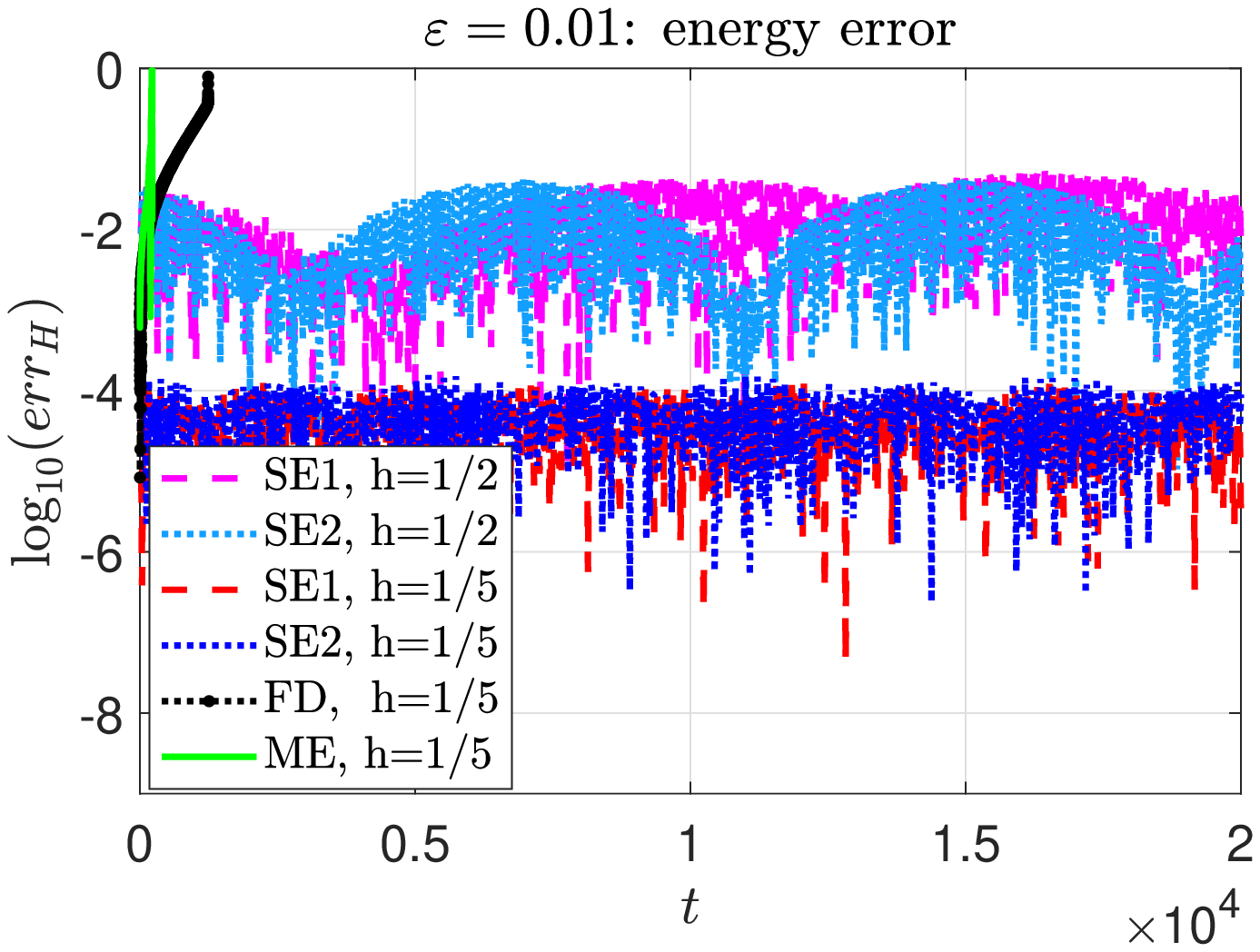,height=3.5cm,width=7.5cm}
\psfig{figure=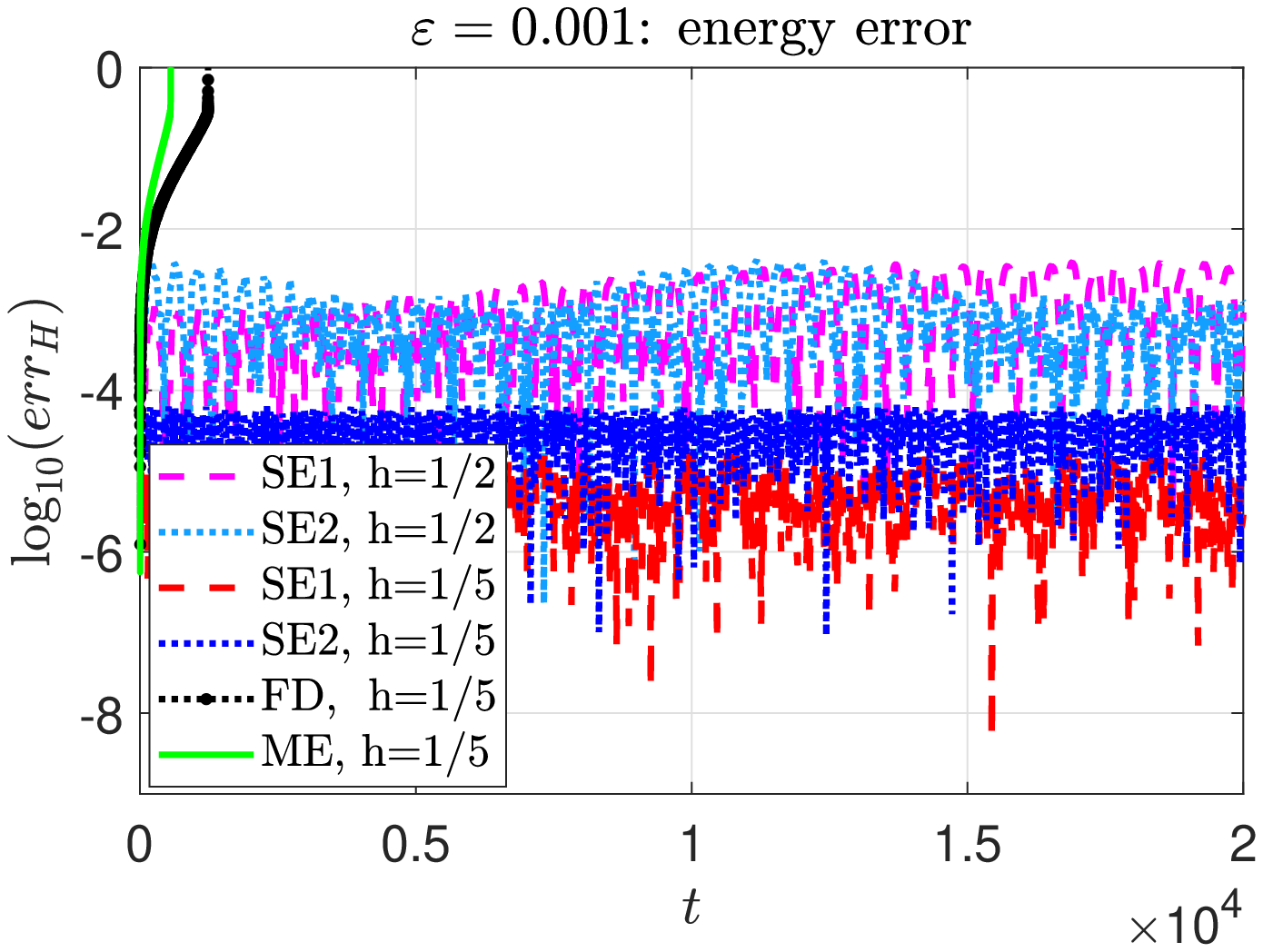,height=3.5cm,width=7.5cm}\\
\psfig{figure=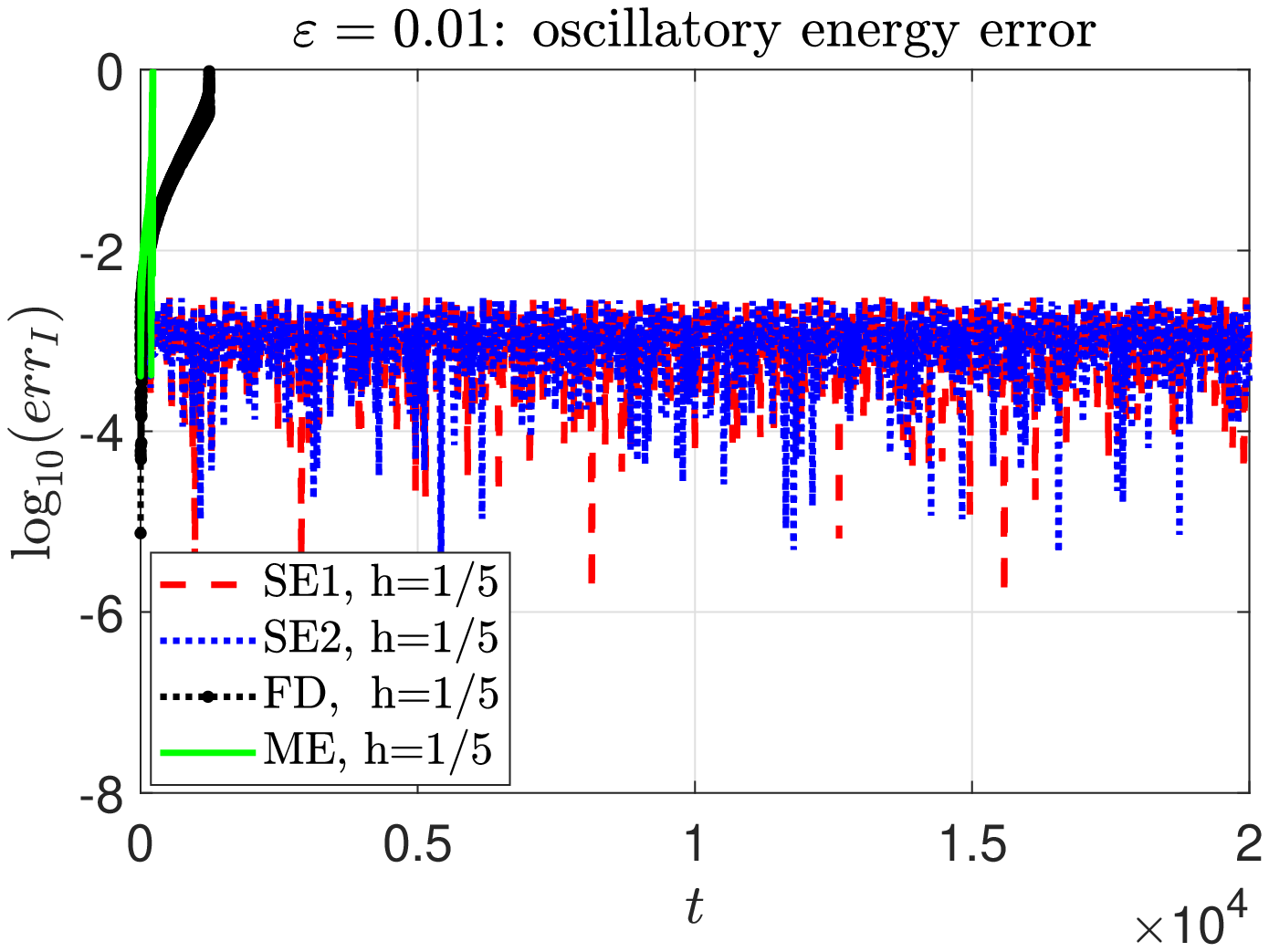,height=3.5cm,width=7.5cm}
\psfig{figure=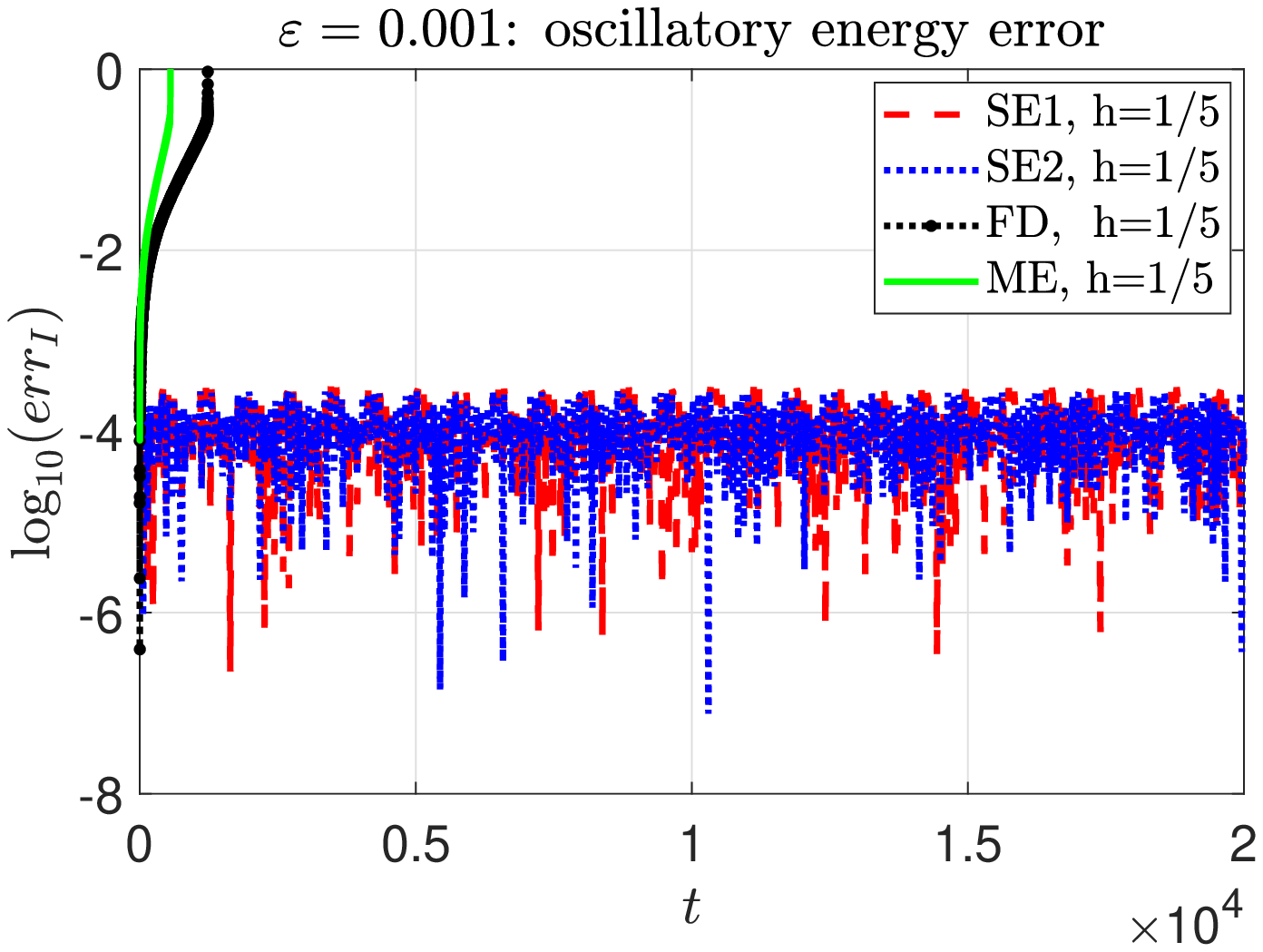,height=3.5cm,width=7.5cm}\\
\psfig{figure=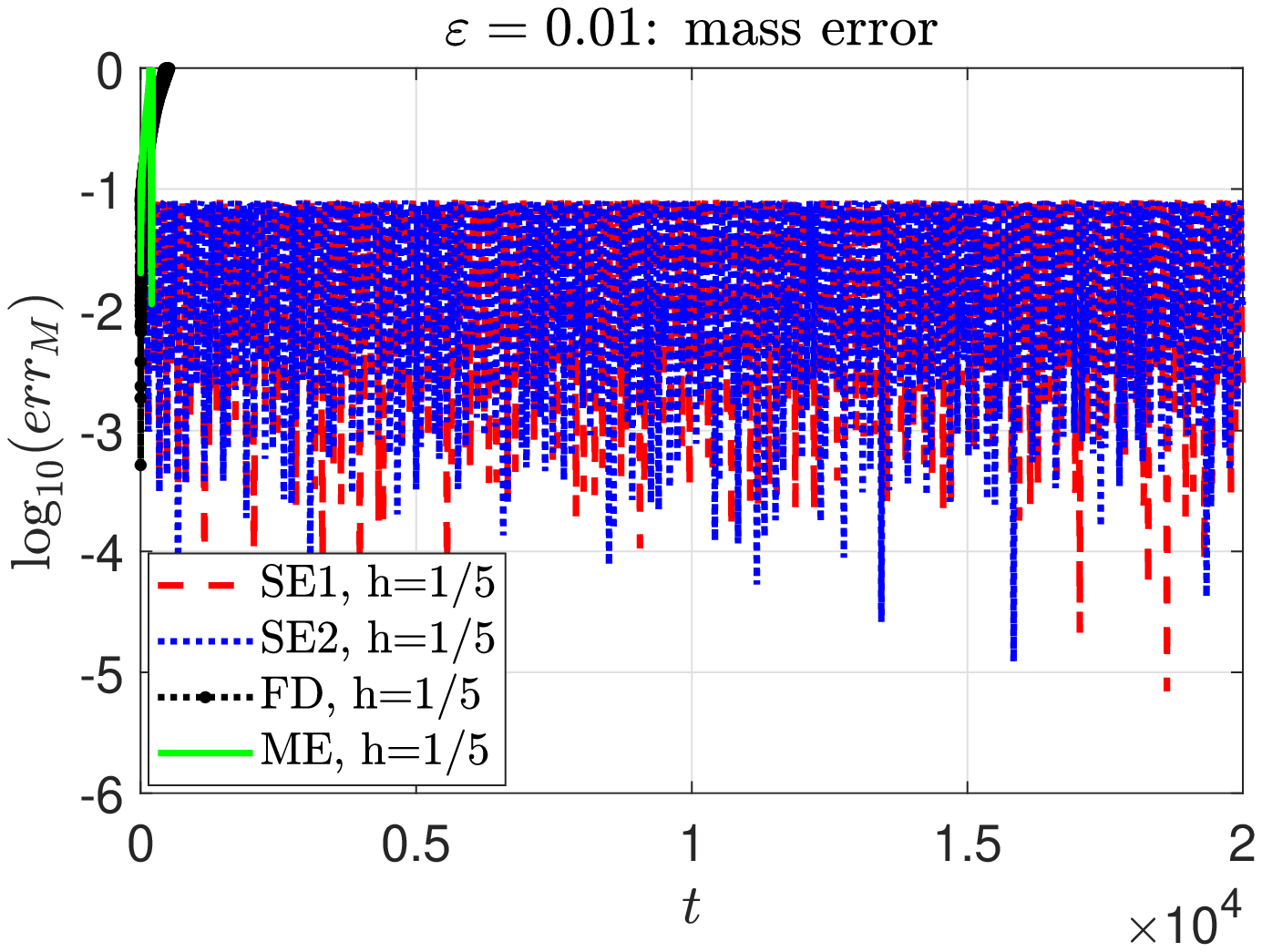,height=3.5cm,width=7.5cm}
\psfig{figure=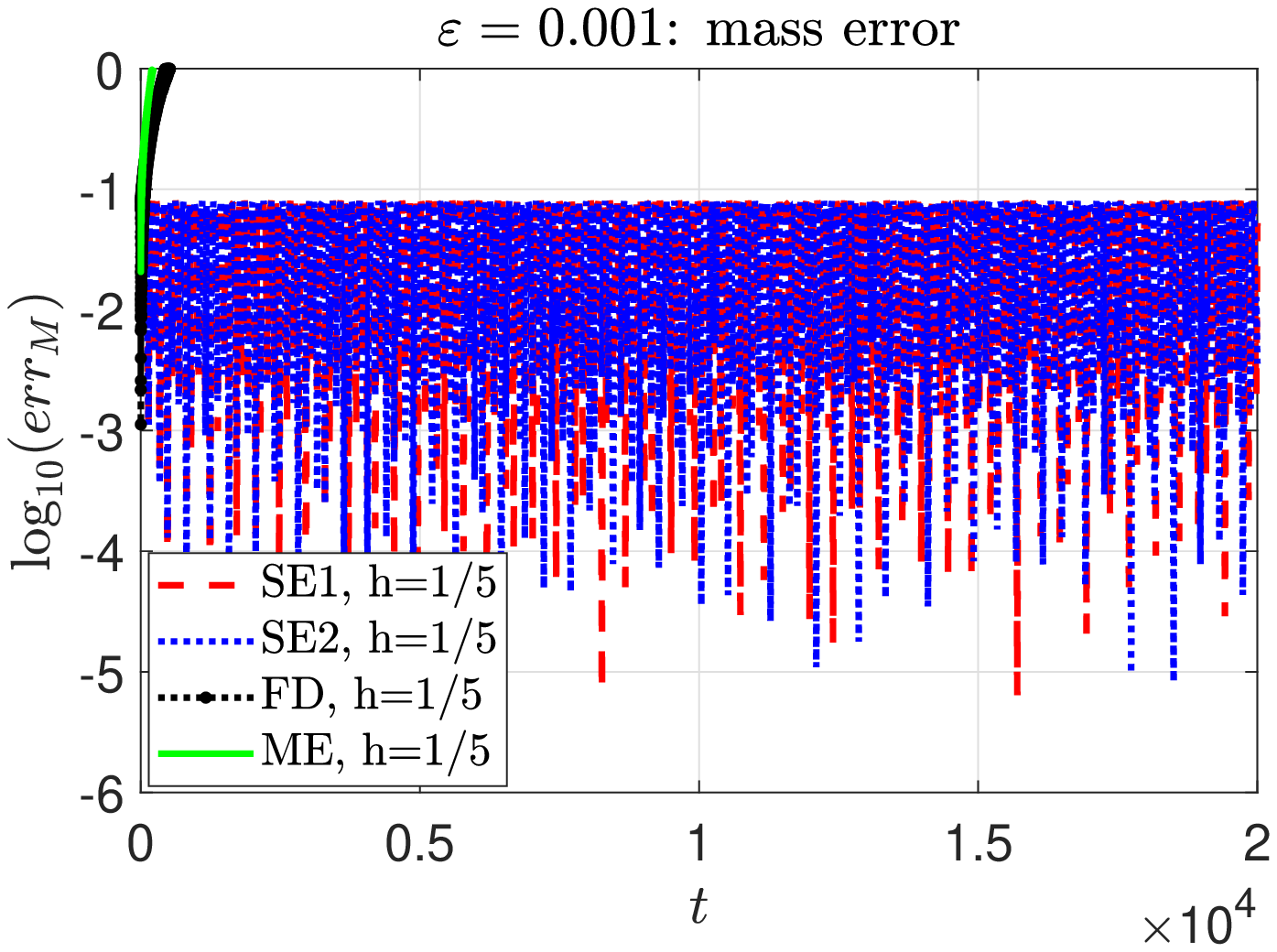,height=3.5cm,width=7.5cm}\\
\end{array}$$
\caption{H\'{e}non-Heiles example (\ref{HH model}): numerical conservations of the schemes  under different $\eps$.}\label{fig11}
\end{figure}

We first test the accuracy of our SE-TSIs by displaying the global errors $err=\abs{u^n-u(t_n)}/\abs{u(t_n)}$ at $t_n=t=1$ in Figure \ref{fig10} for different values of $\eps$. Clearly, the two SE-TSIs are both uniformly second order accurate for all $\eps\in(0,1]$,  which verifies the convergence result in Proposition \ref{prop2}.
Then, we study the long-term performance of the methods by investigating their numerical energy $H(u^n)$ as defined in (\ref{energy}) on a large time interval. In addition to the exact invariant $H(u)$, motivated by \cite{Hairer,Lubich}, we also test the numerical conservations of the oscillatory energy $I(u)$ and the mass $m(u)$:
\begin{equation}\label{mass}
I(u):={\frac{1}{2\eps}}\langle Mu,u\rangle,\quad m(u) :=\langle u,u\rangle,
\end{equation}
 which could be the almost-invariants of the system (\ref{model}).
The conservation  errors
$$err_H=\frac{|H(u^n)-H(u^0)|}{|H(u^0)|},\quad err_I=\frac{|I(u^n)-I(u^0)|}{|I(u^0)|},\quad err_M=\frac{|m(u^n)-m(u^0)|}{|m(u^0)|}$$
of the two SE-TSIs are shown in Figure \ref{fig11} for two different $\eps$. For the exact invariant $H$, we test two time steps: $h=1/2$ and $h=1/5$.
For comparisons, the results of the FD-TSF (\ref{FD scheme}) and the ME-TSF (\ref{EI scheme}) are also shown under $h=1/5$.
According to the numerical results in Figure \ref{fig11}, we have the following observations.

a) The energy $H$, oscillatory energy $I$ and mass $m$ are nearly preserved numerically by both SE-TSIs over long times even for large time step $h$. With smaller time step $h$, the numerical error in the energy can be improved uniformly in $\eps$ (the 1st row of Figure \ref{fig11}).
In contrast, FD-TSI and ME-TSI show substantial drifts in all the quantities in a short time.

b) The oscillatory energy $I$ and mass $m$ are almost invariants of the model. Their conservation errors would depend on both $h$ and $\eps$. It can be seen that as $\eps$ decreases (the 2nd row of Figure \ref{fig11}), the errors of SE-TSF in the oscillatory energy decreases at rate $\mathcal{O}(\eps)$.



   \begin{figure}[t!]
$$\begin{array}{cc}
\psfig{figure=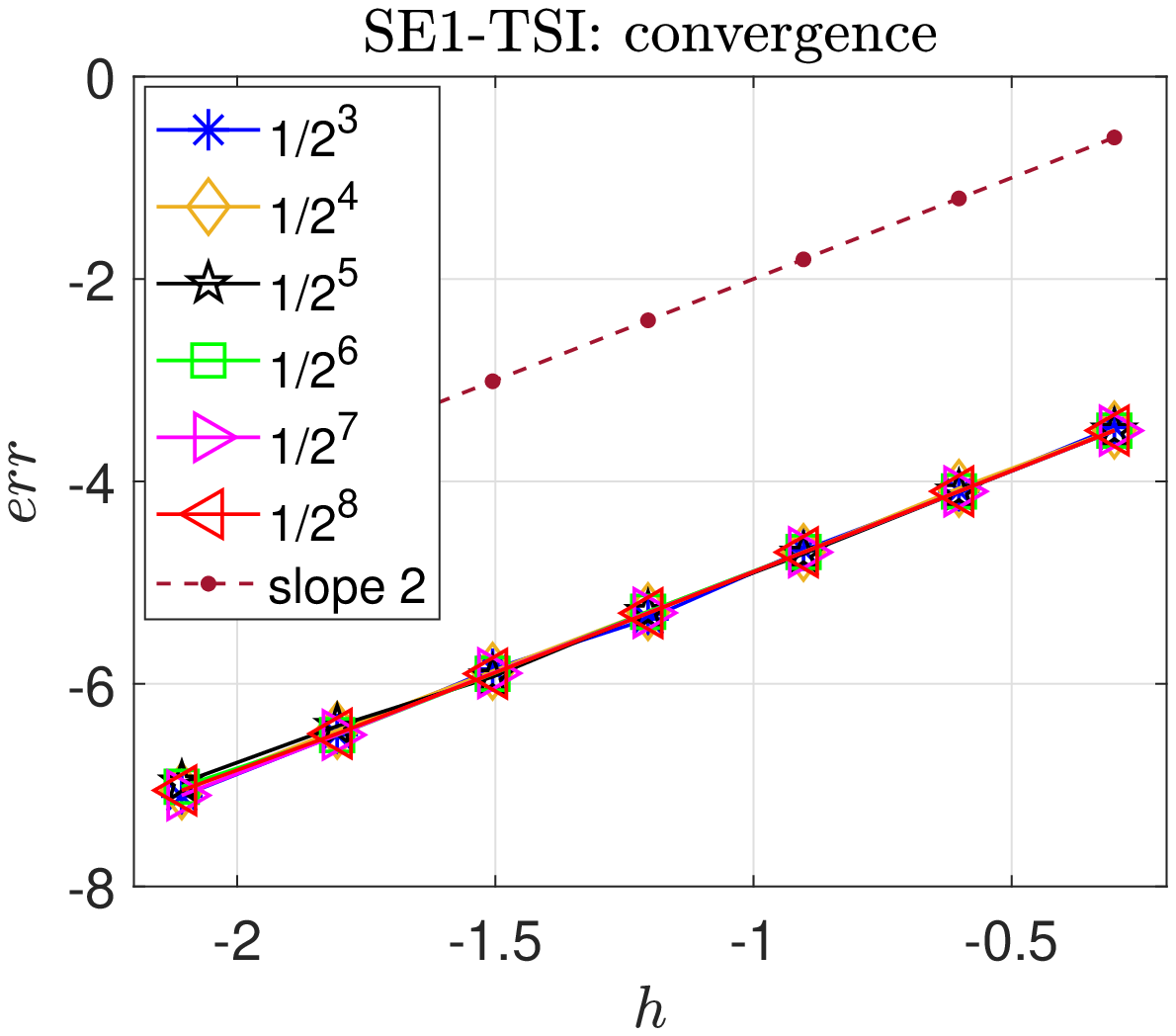,height=4.5cm,width=6.7cm}
\psfig{figure=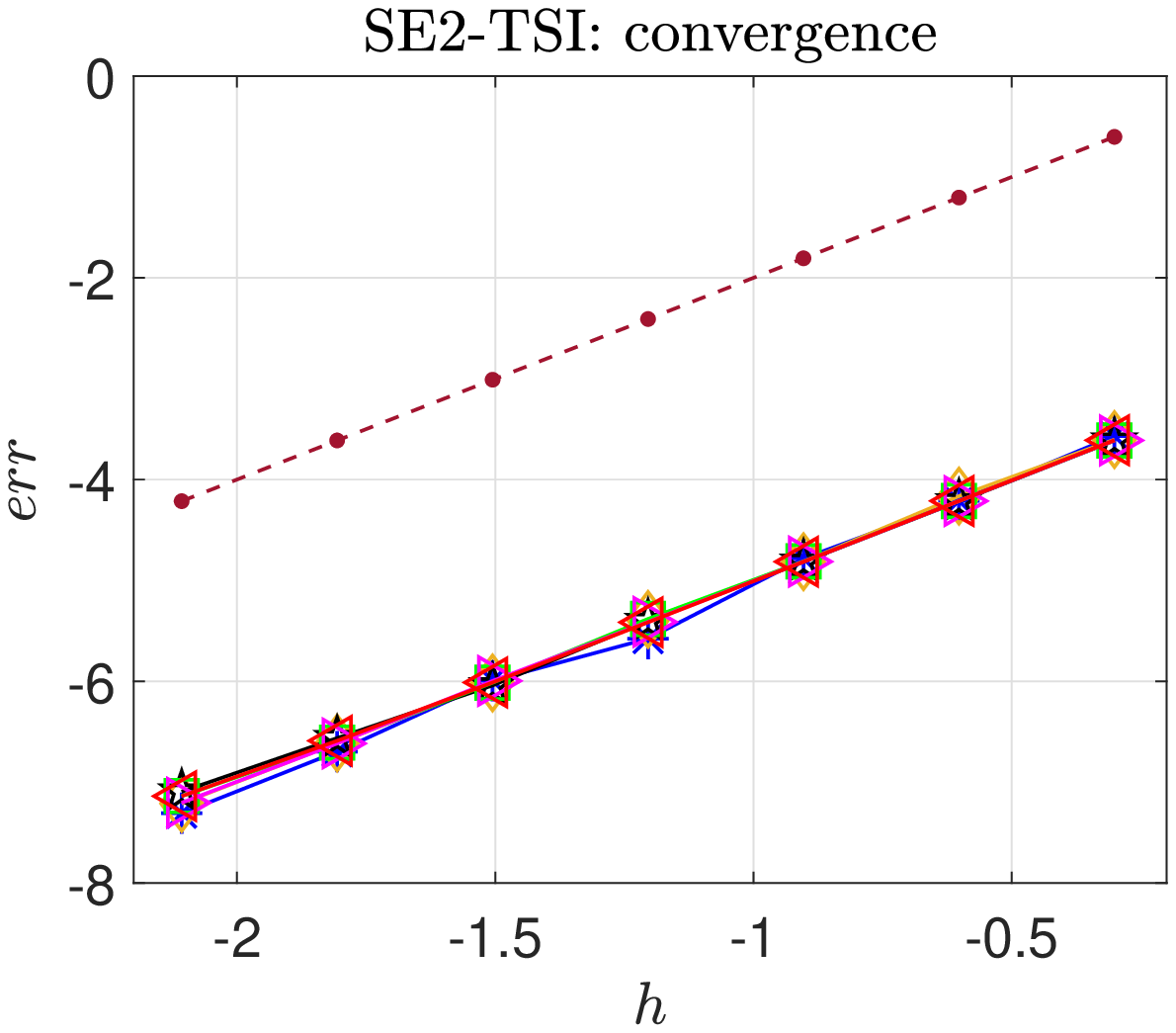,height=4.5cm,width=6.7cm}
\end{array}$$
\caption{NLS example (\ref{NLS eq}): the log-log plot of the temporal error $err=\norm{u^n-u(t_n)}_{L^2}/\norm{u(t_n)}_{L^2}$ of SE-TSIs at $t_n=1$ under different $\eps$.}\label{fig20}
\end{figure}
\begin{figure}[t!]
$$\begin{array}{cc}
\psfig{figure=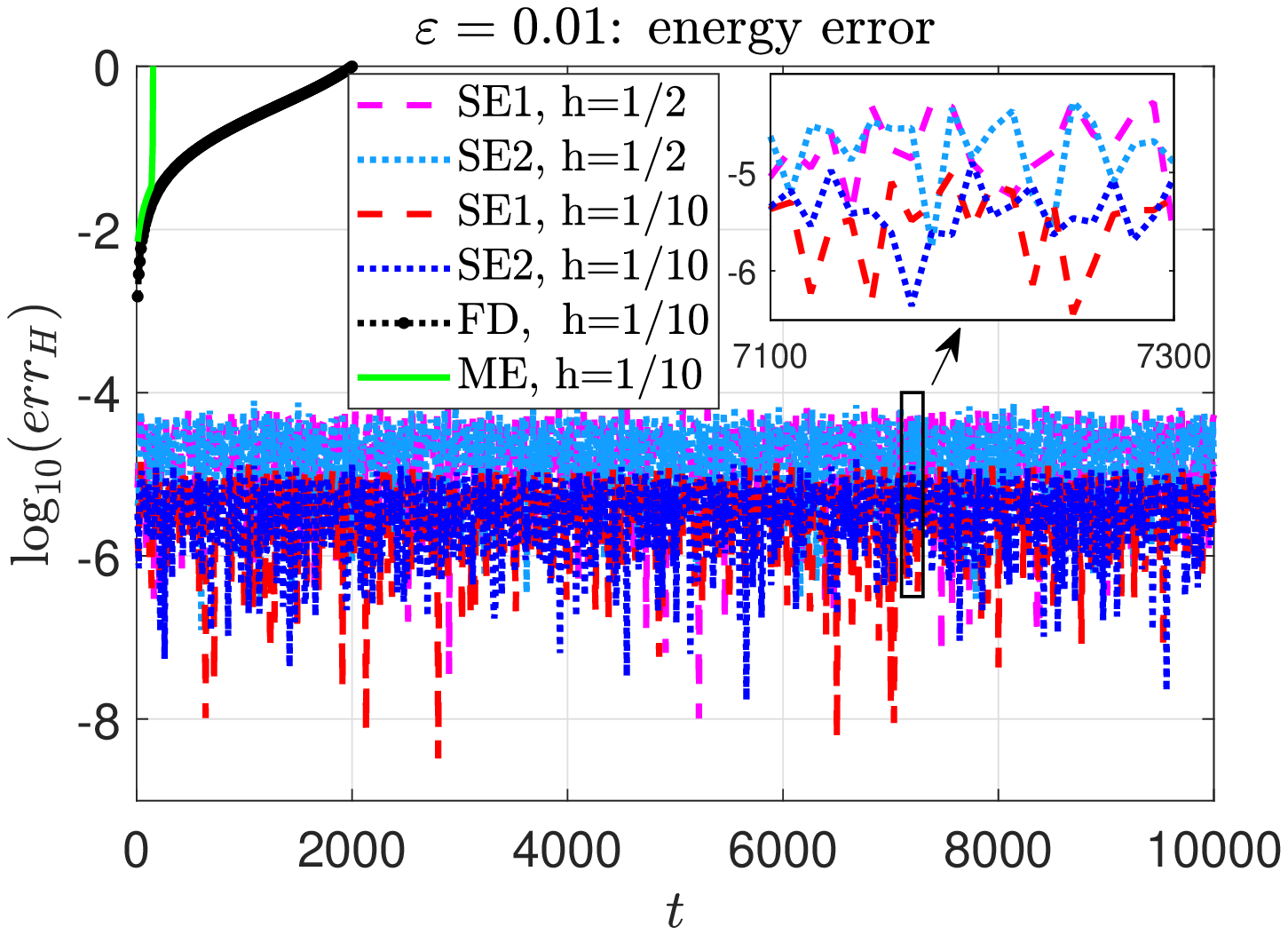,height=3.5cm,width=7.5cm}
\psfig{figure=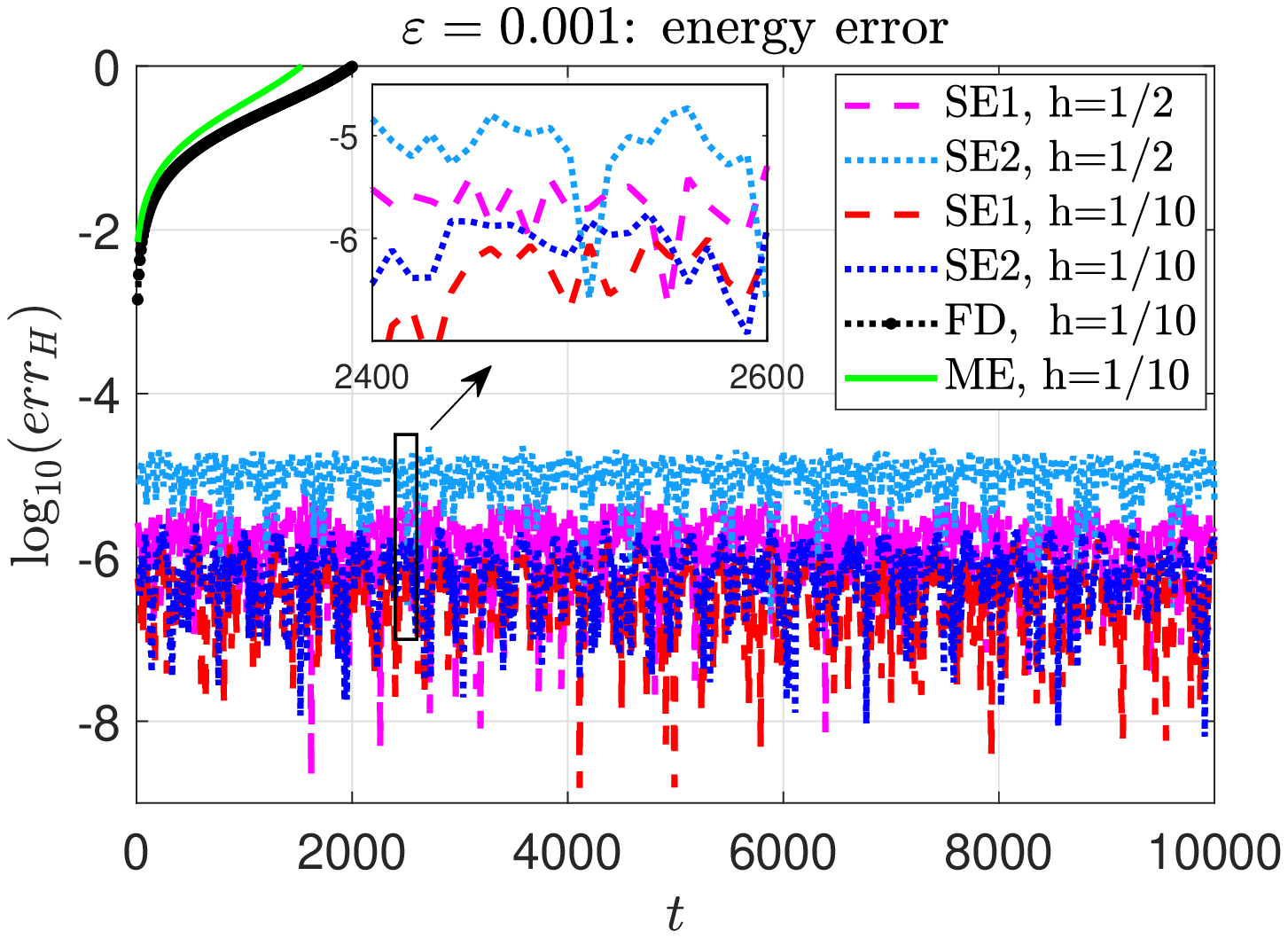,height=3.5cm,width=7.5cm}\\
\psfig{figure=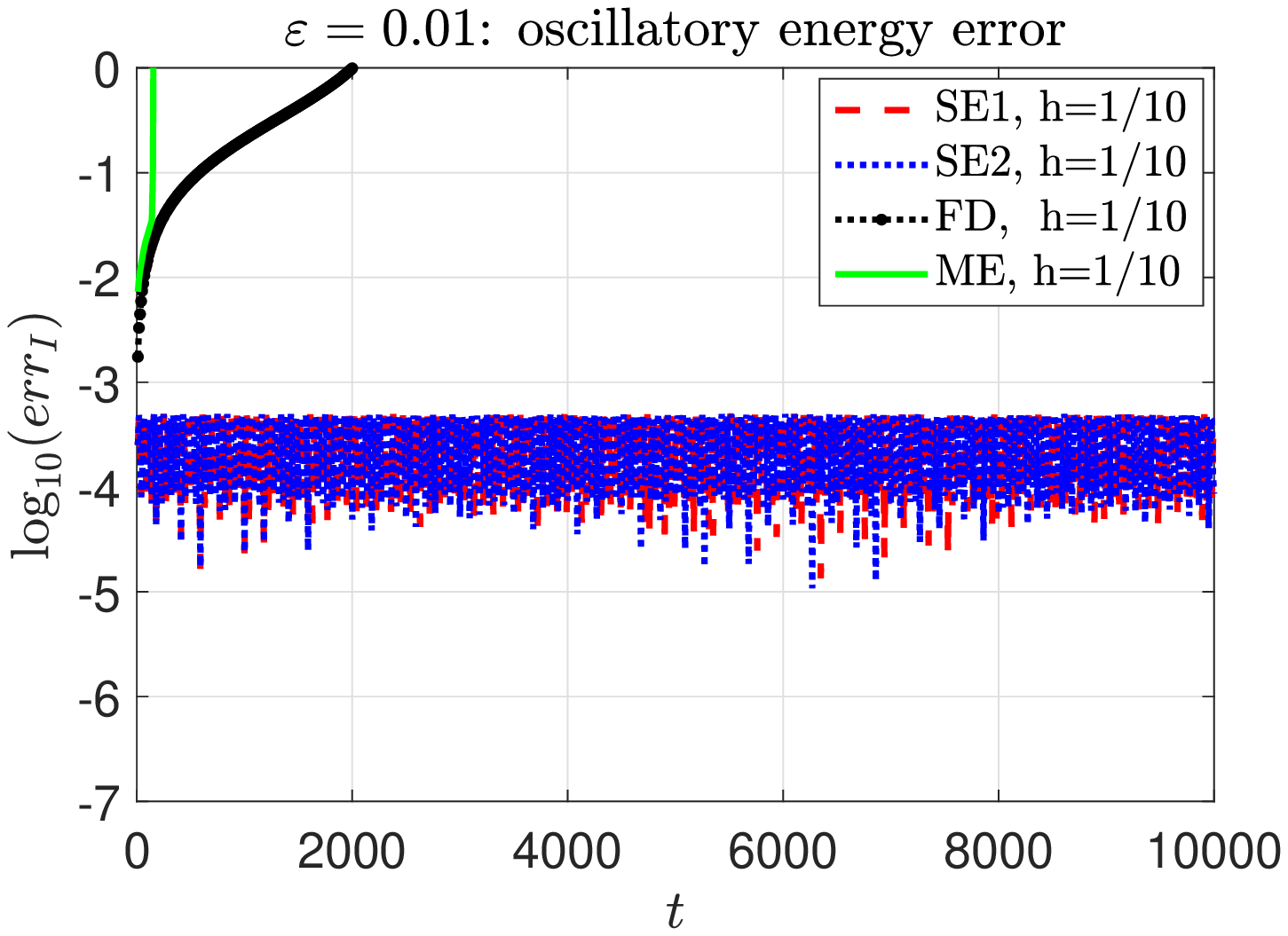,height=3.5cm,width=7.5cm}
\psfig{figure=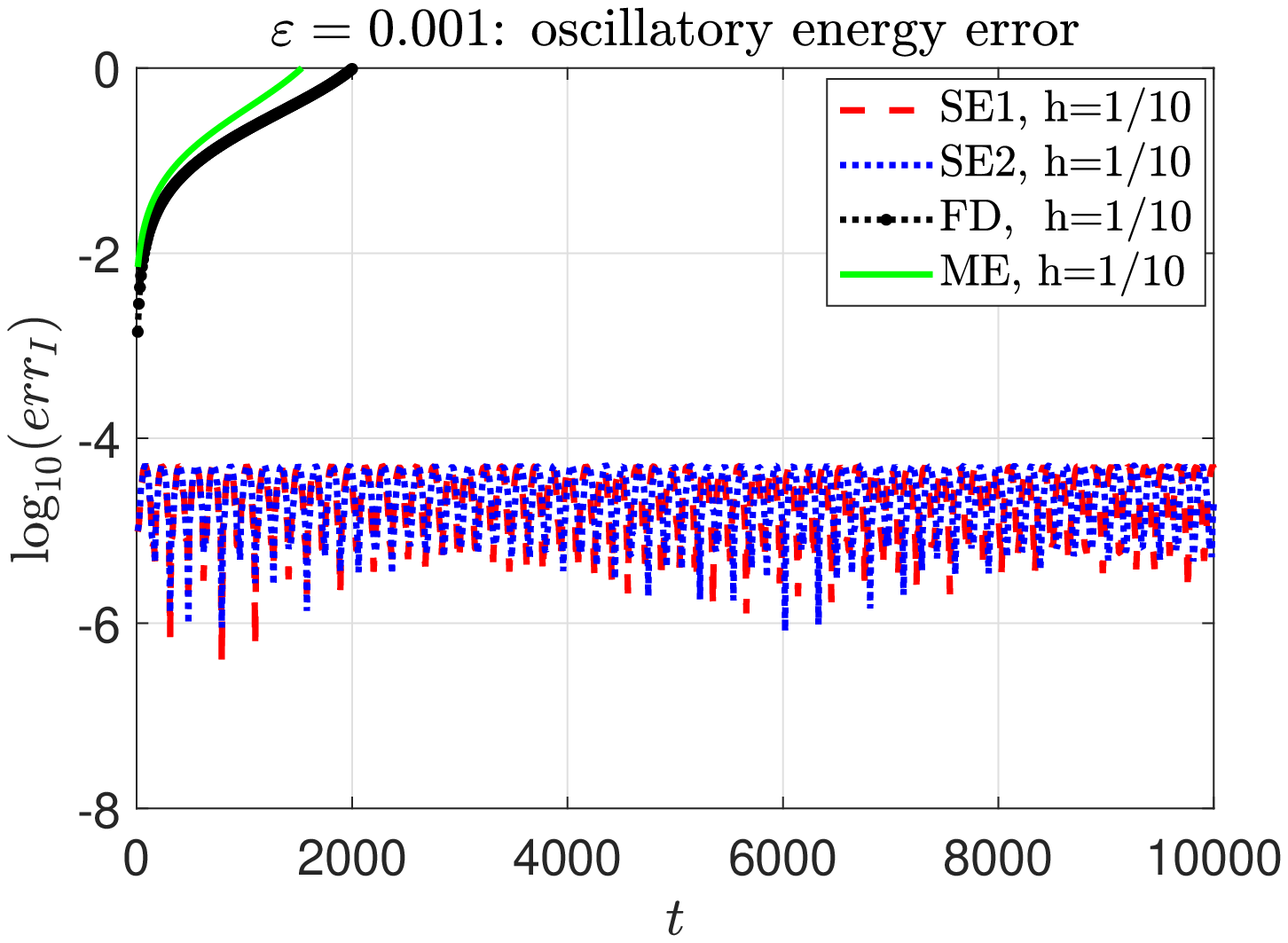,height=3.5cm,width=7.5cm}\\
\psfig{figure=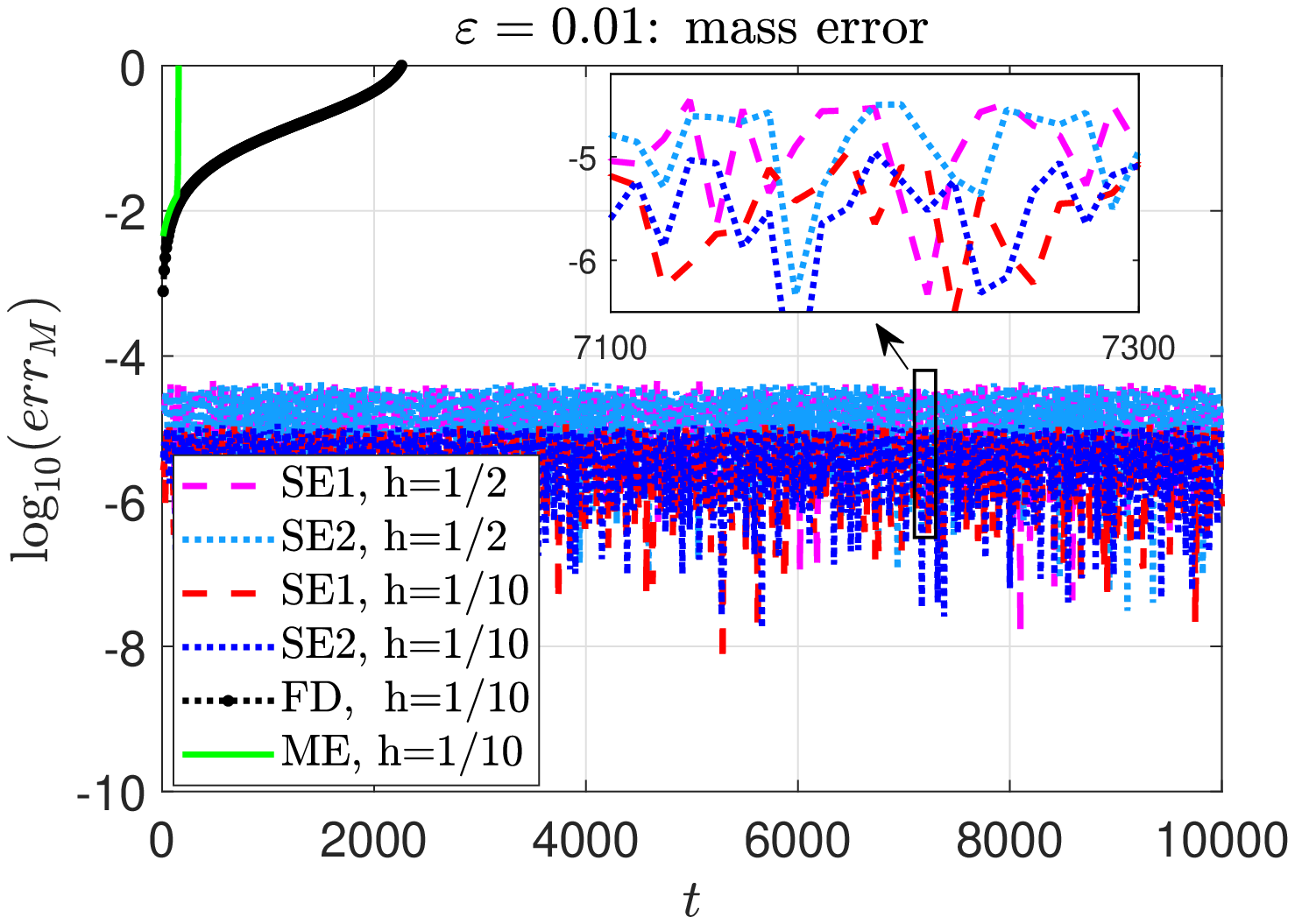,height=3.5cm,width=7.5cm}
\psfig{figure=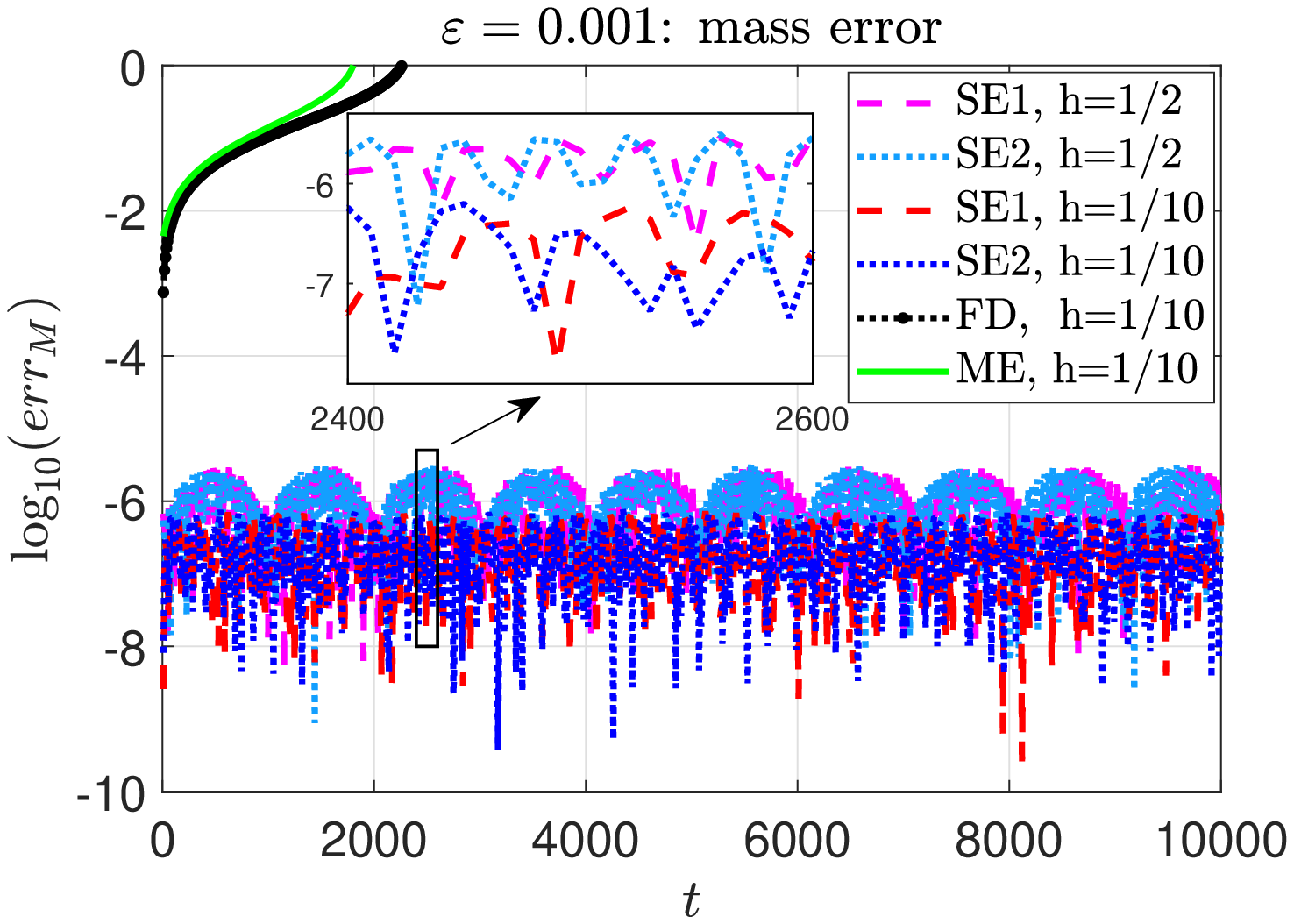,height=3.5cm,width=7.5cm}\\
\end{array}$$
\caption{NLS example (\ref{NLS eq}): numerical conservations of the schemes  under different $\eps$.}\label{fig21}
\end{figure}
  \subsection{Nonlinear Schr\"odinger equation}
Next, we consider the following highly oscillatory cubic nonlinear Schr\"odinger (NLS) equation on the torus \cite{UAKG,SAV}:
  \begin{align}\label{NLS eq}
   i\partial_t u(t,x)=-\frac{1}{\eps}\partial_x^2 u(t,x)+|u(t,x)|^2u(t,x),\quad x\in\bT,\ t>0.
  \end{align}
The model is globally well-posed in $L^2(\bT)$ \cite{Bourgain}, and
it comes by rescaling the long-term dynamics of a perturbed Schr\"odinger equation \cite{Faou,Kuksin} from a time interval of $\mathcal{O}(1/\eps)$ to $\mathcal{O}(1)$. As a PDE, (\ref{NLS eq}) fits into our model problem (\ref{model}) and all the assumptions. The operator $\fe^{i\partial_x^2\tau}$ is periodic in $\tau$ owning to the periodic setup in $x$.
In fact, by a change of variable
$v=\fe^{-\frac{it}{\eps}\partial_x^2}u,$
 (\ref{NLS eq}) reads:
$\partial_t v=-i\fe^{-it\partial_x^2/\eps}[|\fe^{it\partial_x^2/\eps}v|^2
\fe^{it\partial_x^2/\eps}v].$
Then by the two-scale technique, we consider for the augmented solution $U=U(t,\tau,x)$ satisfying
  \begin{align}
&\partial_t U+\frac{1}{\eps}\partial_\tau U=-i\fe^{-i\tau\partial_x^2}\left[\left|\fe^{i\tau\partial_x^2}U\right|^2
\fe^{i\tau\partial_x^2}U\right],\quad x\in\bT,\ \tau\in\bT,\ t>0.\label{NLS eq 2scale}
\end{align}

Now we apply the two-scale methods to solve (\ref{NLS eq 2scale}) with Fourier pseudospectral discretizations in both $x$ and $\tau$.
We respectively choose $N_{x}=2^4$ and $N_\tau=2^8$ to discretize the $x$ and $\tau$ directions which is enough to neglect their discretization errors.  The initial value of (\ref{NLS eq}) is chosen as {$u_0(x)=\frac{\cos(x)+i \sin(x)}{1+\sin^2(x)}$}, and the reference solution is obtained by the Strang splitting scheme \cite{Faou}. The solution errors ($L^2$-norm) and  numerical conservations in $H,\,I$ and $m$ (\ref{mass}) are presented in Figures \ref{fig20}\&\ref{fig21},  respectively. Since the mass $m$ is a well-known exact invariant in the NLS model, so here additionally we test the two time steps in the mass conservation errors. The results (the 3rd row of Figure \ref{fig21}) clearly show the improvement of the errors as $h$ decreases, and the other numerical phenomena are similar as before.

\begin{figure}[t!]
$$\begin{array}{cc}
\psfig{figure=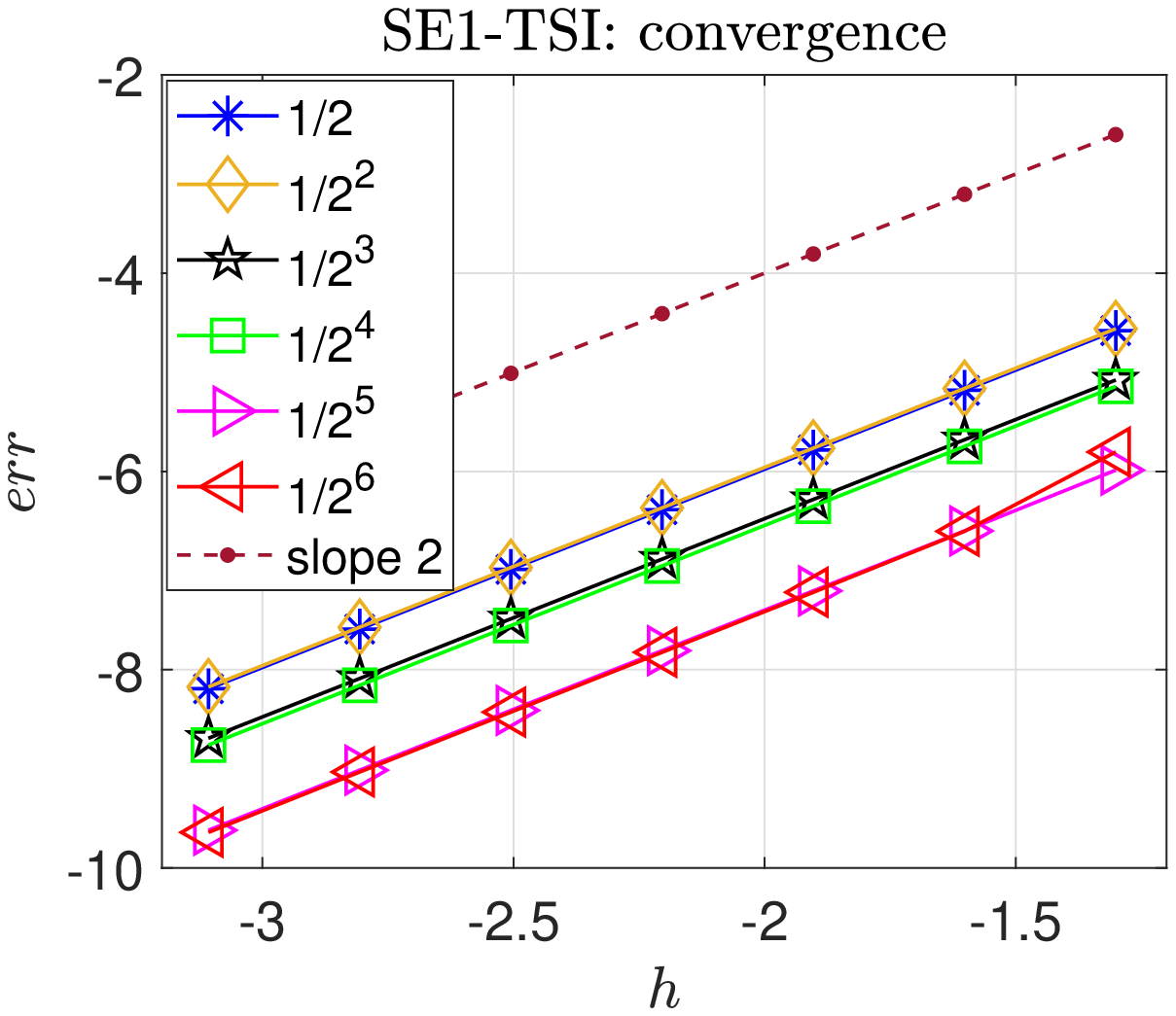,height=4.5cm,width=6.7cm}
\psfig{figure=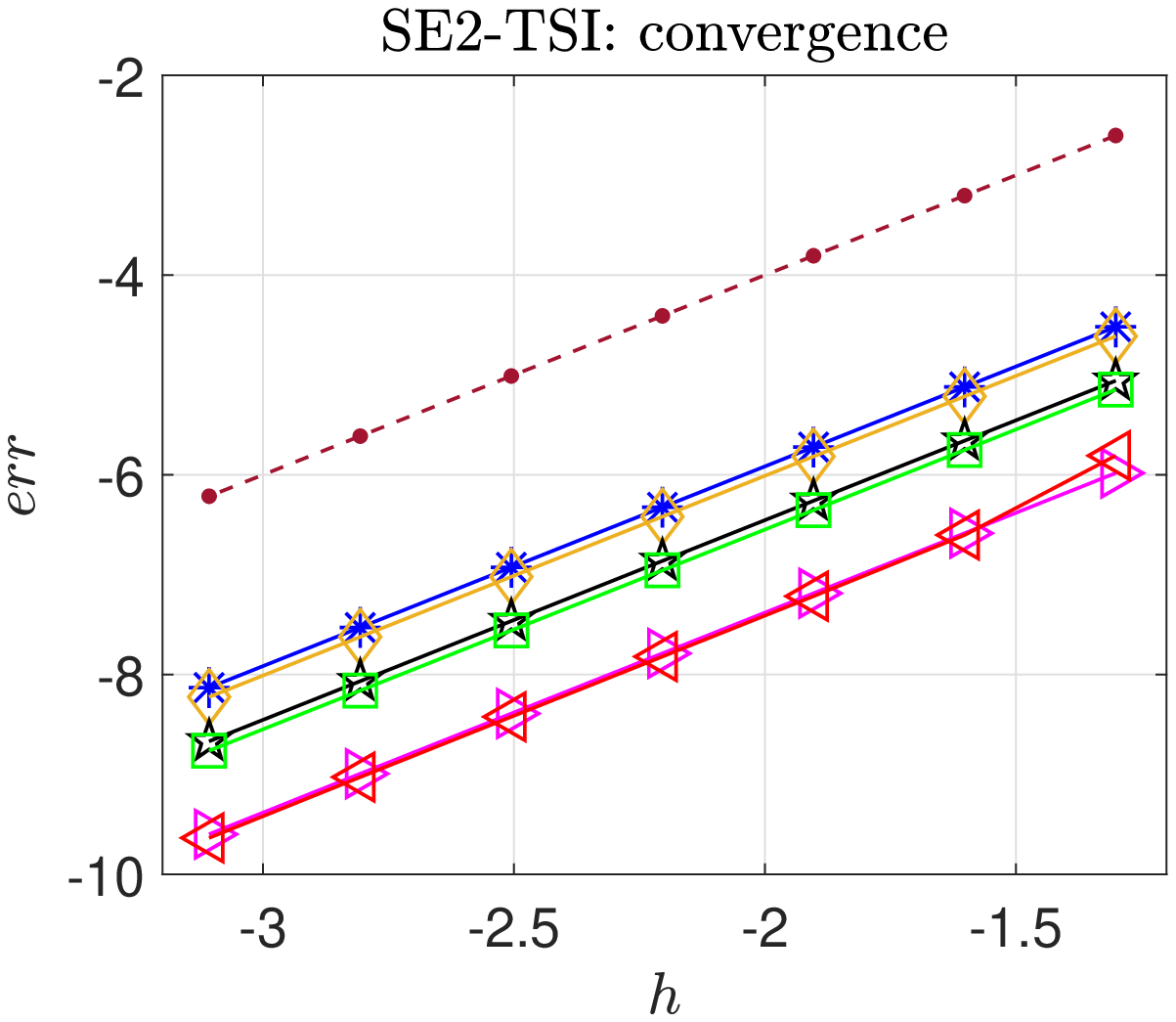,height=4.5cm,width=6.7cm}
\end{array}$$
\caption{CPD example (\ref{CPD}): the log-log plot of temporal error $err=\abs{u^n-u(t_n)}/\abs{u(t_n)}$ of  SE-TSIs at $t=1$ under different $\eps$.}\label{fig30}
\end{figure}
      \begin{figure}[t!]
$$\begin{array}{cc}
\psfig{figure=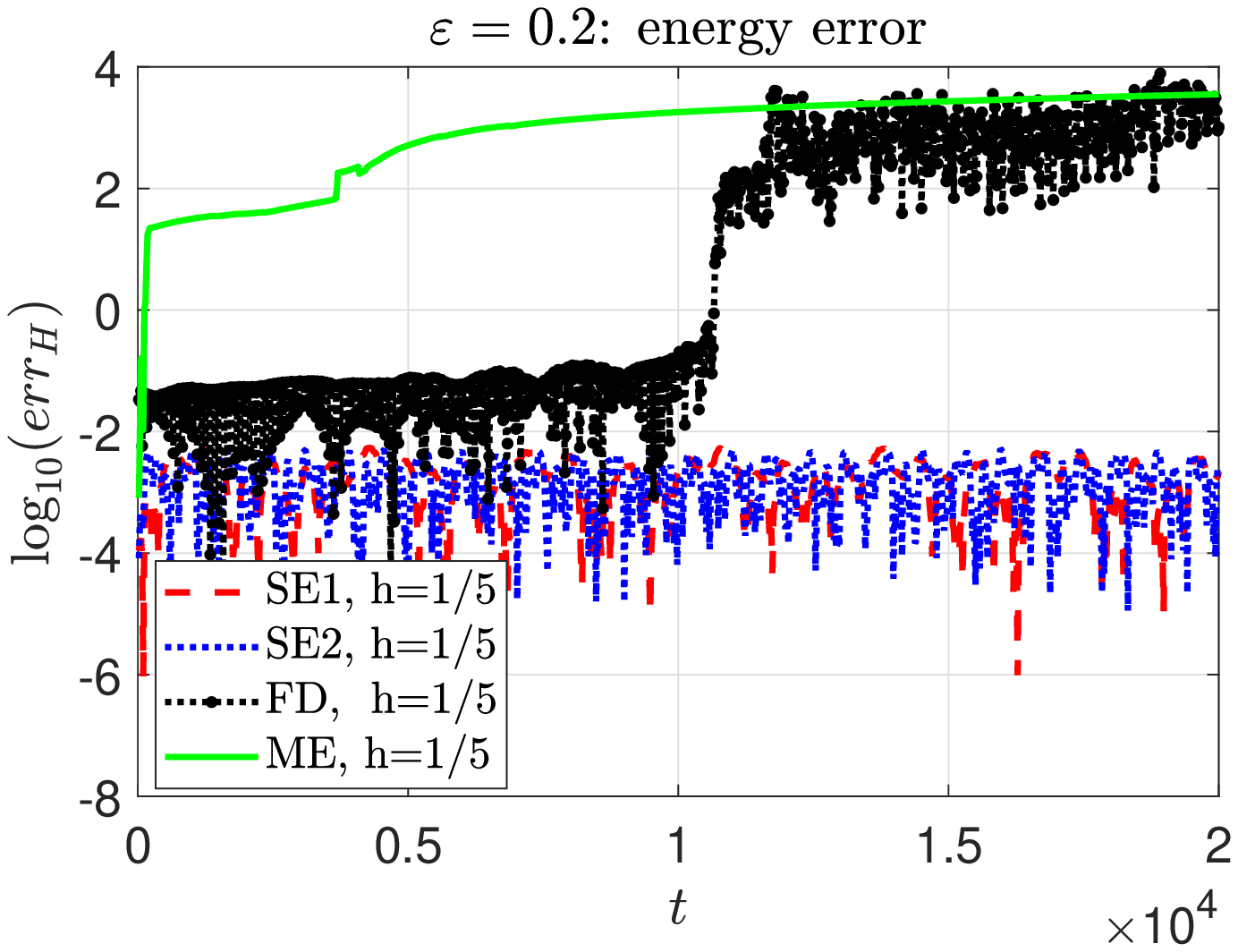,height=3.5cm,width=7.5cm}
\psfig{figure=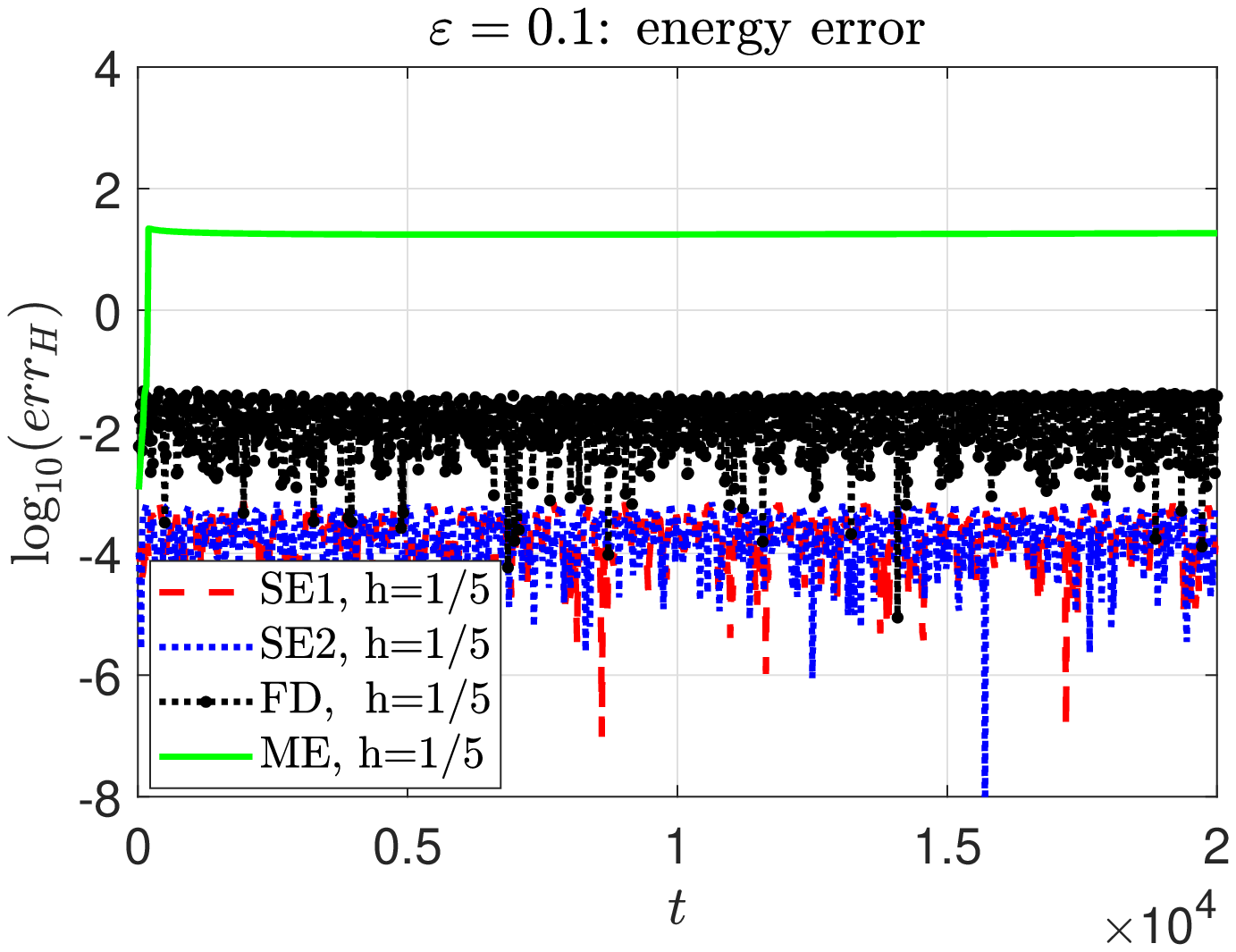,height=3.5cm,width=7.5cm}\\
\psfig{figure=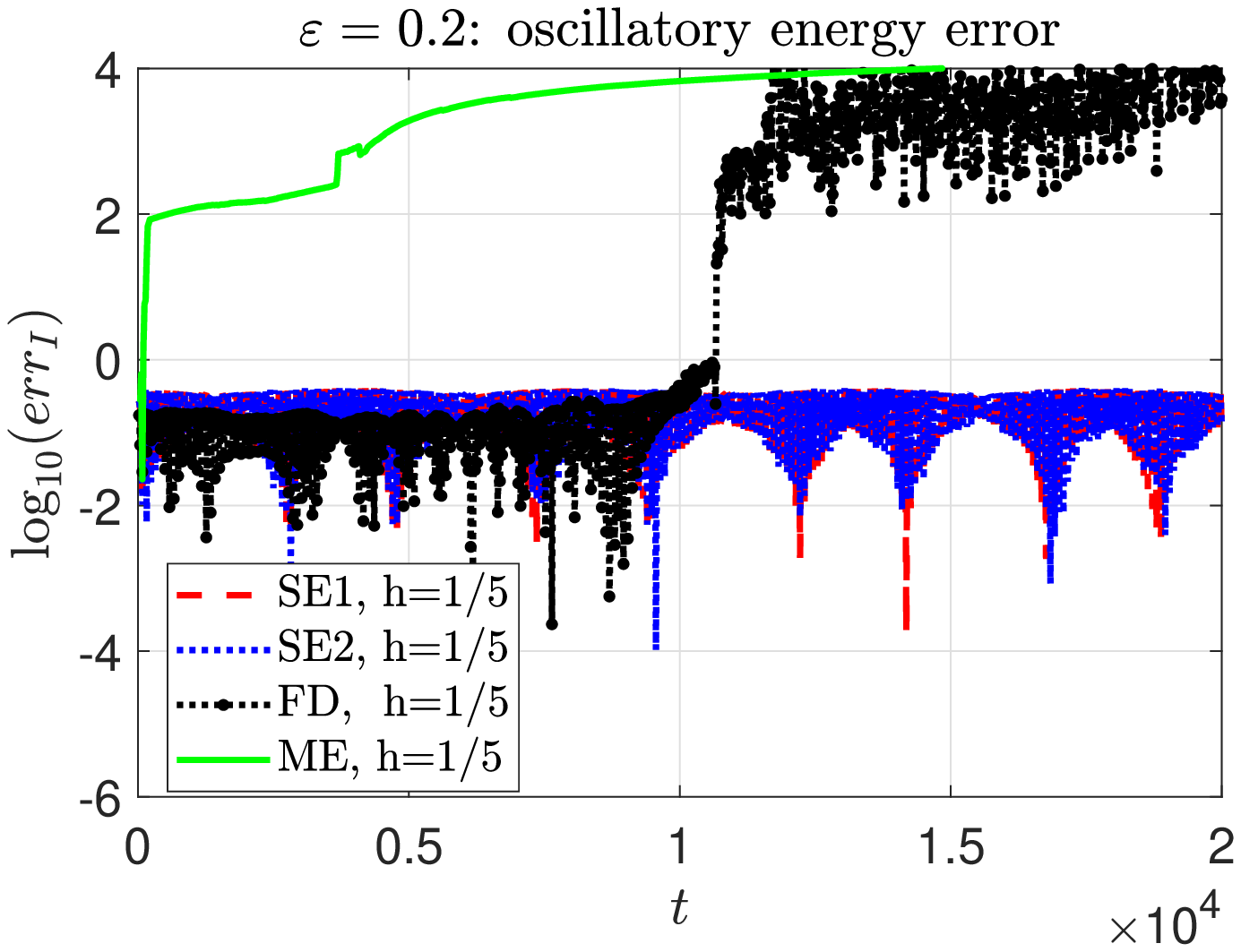,height=3.5cm,width=7.5cm}
\psfig{figure=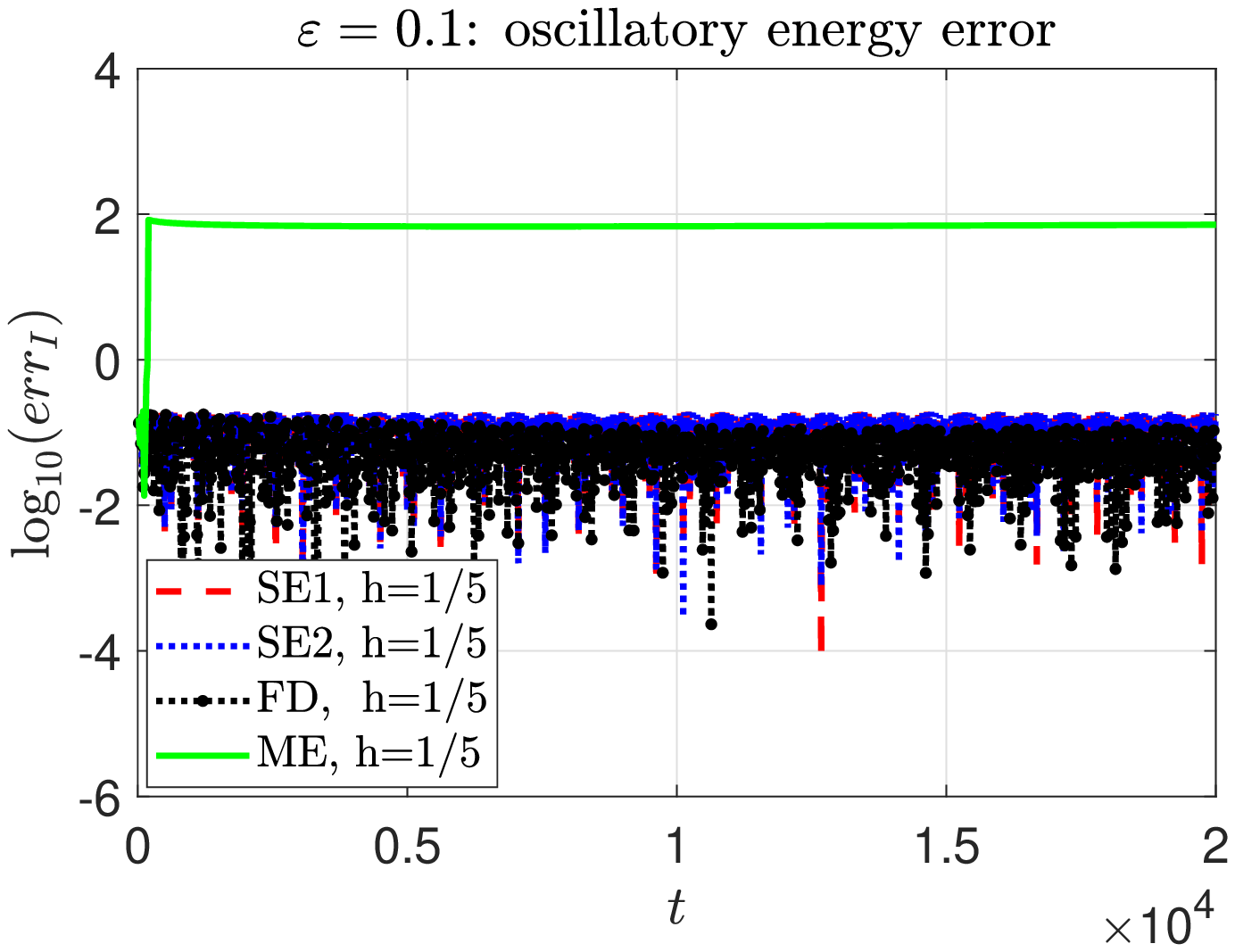,height=3.5cm,width=7.5cm}\\
\psfig{figure=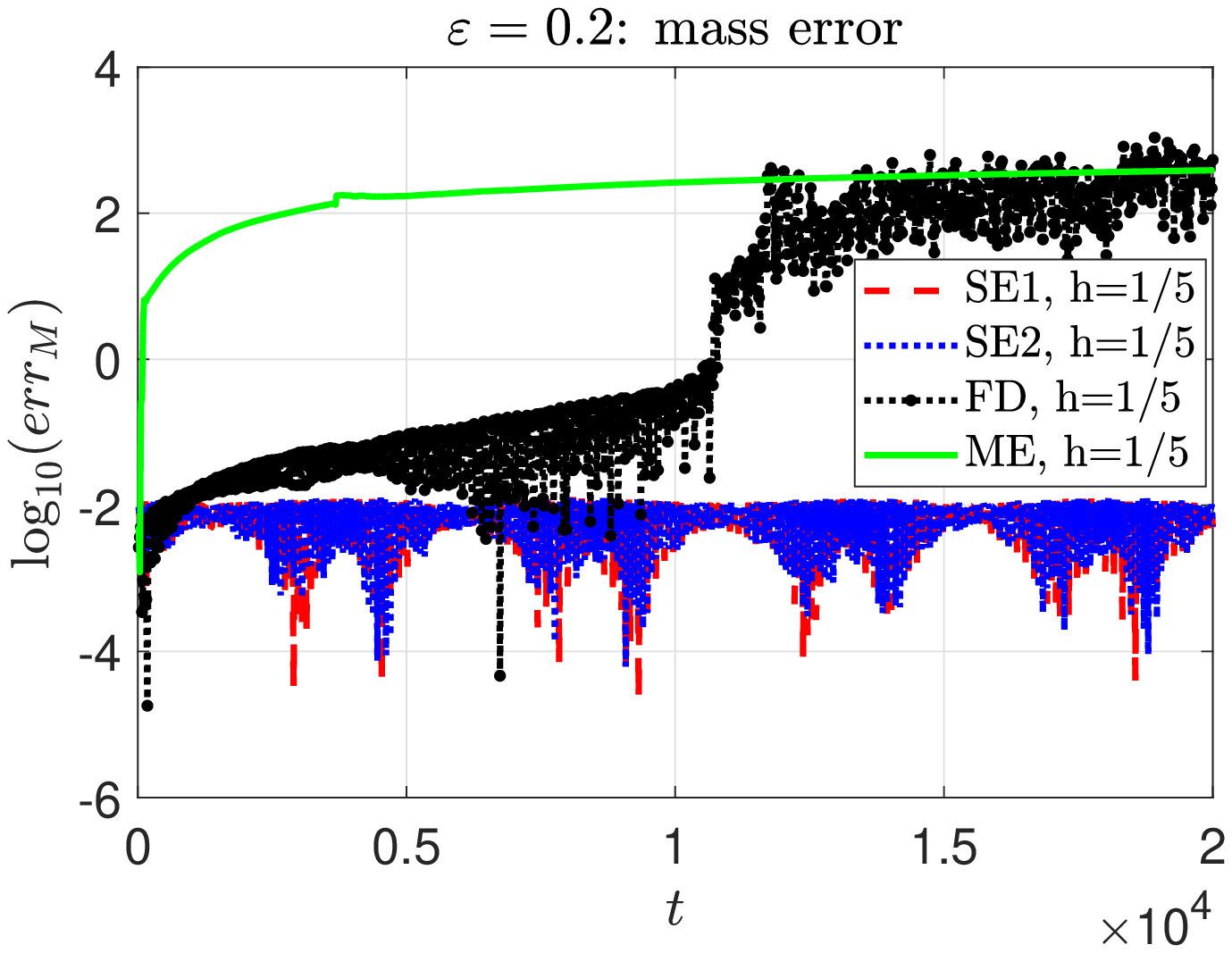,height=3.5cm,width=7.5cm}
\psfig{figure=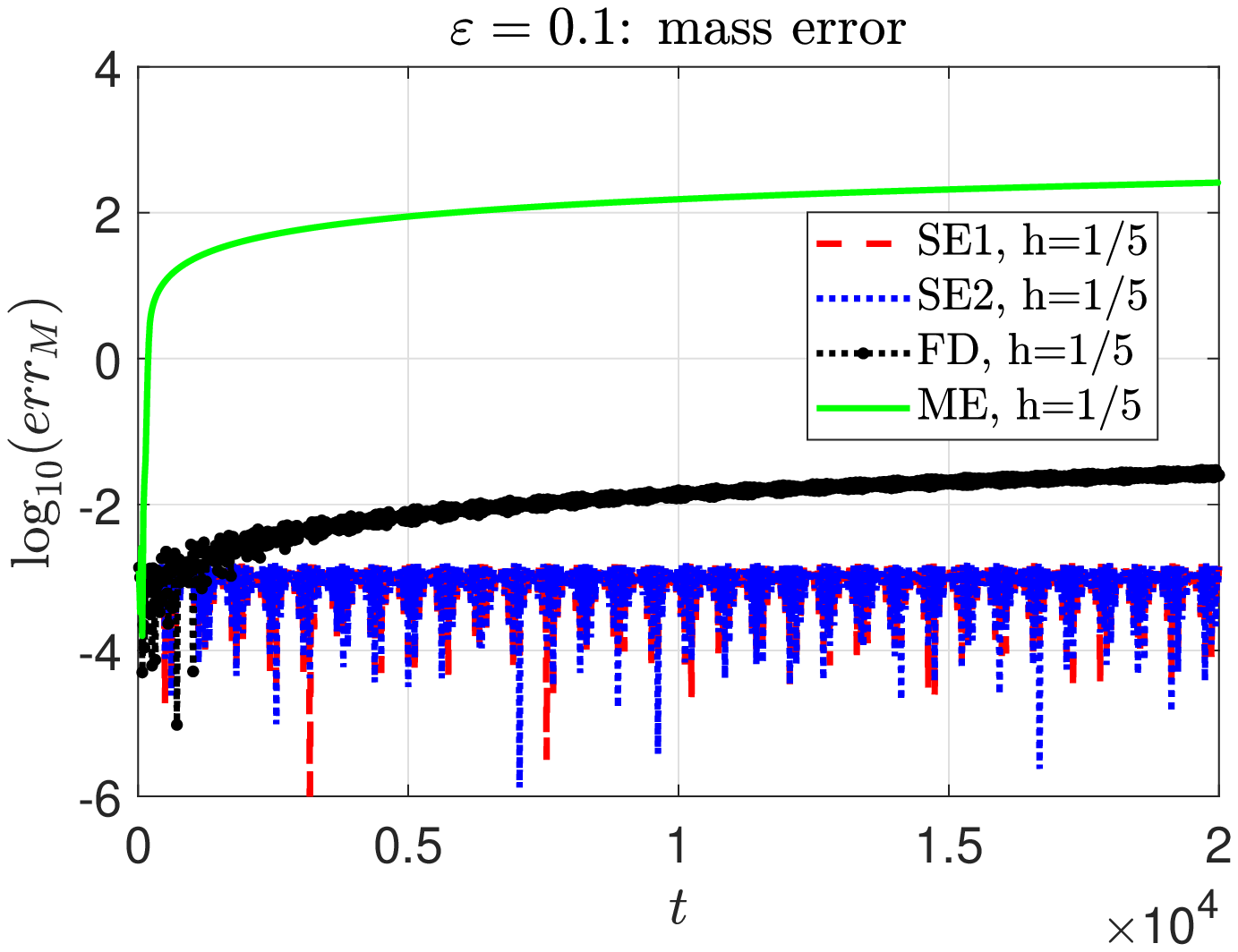,height=3.5cm,width=7.5cm}\\
\end{array}$$
\caption{CPD example (\ref{CPD}): numerical conservations of the schemes  under different $\eps$.}\label{fig31}
\end{figure}

\subsection{Charged-particle dynamics}
 {Last but not least,} we consider
the charged-particle dynamics (CPD) in two space dimensions under a strong magnetic field \cite{Frenod}:
\begin{equation}\left\{
\begin{split}
  &\dot{x}(t)=v(t), \\
  &\dot{v}(t)=\frac{1}{\eps}Bv(t)+g(x(t)),  \quad t>0,
\end{split}\right.\label{CPD}
\end{equation}
where $x(t),v(t)\in\bR^2$ are the unknown, the parameter $\eps\in(0,1]$ is inversely proportional to the strength of the magnetic field, $g(x)=-\nabla U(x)$ is generated by a given scalar potential {$U(x)=1/\sqrt{x_1^2+x_2^2}$},  and $B=\binom{\ 0\ \,1}{-1\ 0}
$.
It is a reduced model from the three dimensional CPD case \cite{vp3D,WZ} when the external magnetic field has a fixed direction and is homogenous in space.
In (\ref{CPD}), the energy $E(x,v)=\frac{1}{2}\abs{v}^2+U(x)$ is conserved.

We note that the CPD (\ref{CPD}) is not in the canonical form. So in order to solve it in our context, we introduce a change of variable 
 $q(t):=x(t),\ p(t):=2 \eps v(t)-Bx(t)$, and then (\ref{CPD}) can be equivalently  formulated as
\begin{equation}\label{model CPD F}
\frac{\textmd{d}}{\textmd{d} t}\begin{pmatrix}
    q(t) \\
     p(t) \\
   \end{pmatrix}
 =\frac{1}{2\eps}\begin{pmatrix}
0 & I_{2\times2} \\
-I_{2\times2}&0
\end{pmatrix} \begin{pmatrix}
-B^2 & -B \\
B&I_{2\times2}
\end{pmatrix}  \begin{pmatrix}
    q(t) \\
    p(t) \\
   \end{pmatrix}+2\eps\begin{pmatrix}
0 & I_{2\times2} \\
-I_{2\times2}&0
\end{pmatrix}\begin{pmatrix}
    \nabla U(q(t))\\
     0 \\
   \end{pmatrix}.
\end{equation}
The above equation fits \eqref{model} with $J=\binom{\ \ 0\ \ \ I_{2\times2}}{-I_{2\times2}\ \ 0}$ and $M= \frac{1}{2}\binom{-B^2\ \ -B}{\ B\ \ \ \ I_{2\times2}}= \frac{1}{2}\binom{I_{2\times2}\ -B}{\ B\ \ \ \ I_{2\times2}}$. It is direct to check that $J$ and $M$ commute and the eigenvalues of $JM$ are $0,\,\pm i$.
The Hamiltonian $H$  and the oscillatory energy $I$ of \eqref{model CPD F} are
$$ H(q,p) :=I(q,p)+2 \eps U(q),\qquad
   I(q,p) :=\frac{1}{4\eps}\begin{pmatrix}
    q \\
     p \\
   \end{pmatrix}^\intercal  \begin{pmatrix}
I_{2\times2} & -B \\
B&I_{2\times2}
\end{pmatrix} \begin{pmatrix}
    q \\
     p \\
   \end{pmatrix},$$ and the mass $m$ is given by $m(q,p) :=q^\intercal q+p^\intercal p.$

Now it is ready to apply the presented methods to \eqref{model CPD F}. Denote by $q^n$ and $p^n$ the numerical solutions of \eqref{model CPD F} obtained from   SE1-TSI \eqref{S1} or SE2-TSI \eqref{S2}, and
then the approximations for  \eqref{CPD} are given by $x^n:=q^n$ and $v^n:=\frac{1}{2\eps}p^n+\frac{1}{2\eps}B q^n$.
We choose the initial value $x(0)=(0.8,0.9)^\intercal,\,v(0)=(0.5,0.6)^\intercal$ for (\ref{CPD}), and we fix $N_\tau=2^6$ in the computations. Figure \ref{fig30} displays the numerical errors in the solution $u=(x,v)^\intercal$, and Figure \ref{fig31} shows the conservation errors in the $H,\,I,\,m$ as before. Here two rather large values of $\eps$ are considered in the tests of Figure \ref{fig31}.
This is because the dynamics in (\ref{CPD}) reduce to zero at the leading order when $\eps\to0$, and then one would not be able to see remarkable differences in the methods for very small $\eps$.
Also due to this special limit situation, as $\eps$ decreases in Figure \ref{fig30}, one can observe some decrease of the error under the same step size. The other numerical observations remain the same as before.

In summary, based on all the numerical results in this section, the UA convergence rate and the  near-conservation laws in long times
 of SE-TSIs are clearly shown. To get rigorous statements on this good long-time behaviour illustrated here, we shall utilize the modulated Fourier expansion to carry out analysis in the next section. To end this section, we remark that our numerical approach can be extended to high order cases.  Applications can also be made to more general models that may not satisfy the assumptions. These will be addressed in a separate work.

\section{Long-term analysis}\label{sec:5LTA}
In this section, we conduct the analysis for the long-time behaviour of the proposed SE-TSI methods and establish their near-conservation laws illustrated numerically in Section \ref{sec:4num}.
To be precise with the model, we
consider here a real-valued finite-dimensional Hamiltonian system case of (\ref{model}) for simplicity,
where  $u(t): [0,t_{\textmd{end}}]\rightarrow \mathbb{R}^{2d}$, $J=\binom{\ \ 0\ \ \ I_{d\times d}}{-I_{d\times d}\ 0}$,  $M$ is a $2d\times 2d$ diagonal\footnote{The diagonal form of $M$ is not essential, since the model \eqref{model}
and the numerical methods are invariant under a diagonalization of the matrix.} matrix independent of $\eps$,
and $H_1(u)$ is a real-valued scalar function.
Moreover, it is assumed that $JM=MJ$.
The conserved Hamiltonian is given by
   \begin{equation}\label{Hcondi}
 H(u(t))=I(u(t))+H_1(u(t))\ \ \textmd{with}\ \     I(u)=\frac{1}{2\eps}u^\textup{T}Mu,
\end{equation}
and the mass of the system is denoted by $ m(u)=u^\textup{T}u.$

 {As an analytical technique firstly introduced in \cite{Hairer}, the modulated Fourier expansion (MFE) has been found powerful to investigate various long-time phenomena in oscillatory ODE/PDE systems and in their corresponding numerical methods. Though this technique has been successfully applied to different numerical methods on several models \cite{Cohen0,Gauckler,Lubich ICM,Hairer,Lubich,LubichW,WZ}, it has not been considered yet for any UA methods on models with unbounded energy. The schemes in such cases, for instance our two-scale method which enlarges the dimension of the original system, are usually more involving than the standard algorithms. The unbounded energy and the complexity of the scheme bring more difficulties in the derivation of the MFE as well as its almost-invariants. Moreover, the existing long-time analysis often require the small initial data condition \cite{Cohen0,Gauckler,Lubich1,Lubich2}. While, we should note that the scale of the initial data here can be independent from the scaled physical parameter $\eps$. We managed to separate the two in our result and analysis below.}

To state the main result and the proof, we first introduce some notations and assumptions.

\subsection{Some preliminaries}\label{subsec1}
As stated in Section \ref{sec:2review}, the two-scale formulation for \eqref{model} is
\begin{equation}\label{2scale0}  \partial_tU(t,\tau)+\frac{1}{\eps}\partial_\tau U(t,\tau)={ \fe^{-\tau JM}J\nabla H_1\left(\fe^{\tau JM}U(t,\tau)\right)},\quad t>0,\ \tau\in\bT,
\end{equation}
where  $U(0,\tau)$ is given by  \eqref{2nd data}.
{Notice that
$$\partial_U H_1\left(\fe^{\tau JM} U(t,\tau)\right)=\big(\fe^{\tau JM}\big)^\intercal \nabla H_1\left(\fe^{\tau JM}U(t,\tau)\right)=
\fe^{-\tau JM}\nabla H_1\left(\fe^{\tau JM}U(t,\tau)\right),$$ and
$$\fe^{-\tau JM}J\fe^{\tau JM}=J\fe^{-\tau JM}\fe^{\tau JM}=J.$$ The formulation \eqref{2scale0} is then given by
\begin{equation}\label{2scale} \partial_tU(t,\tau)+\frac{1}{\eps}\partial_\tau U(t,\tau)=J\partial_U H_1\left(\fe^{\tau A} U(t,\tau)\right),\quad t>0,\ \tau\in\bT,
\end{equation}
with $A=JM$}.  With some even integer $N_\tau>0$,
the $\tau$-direction is discretized by the Fourier pseudospectral method \cite{Shen,Trefethen} under mesh size $\Delta\tau=2\pi/N_\tau$. With $\widehat{U}_\ell$ the discrete Fourier coefficient of $U$, we denote the solution after the semi-discretiation (in $\tau$) as
 $$ \widetilde{U}(t,\tau):=\sum\limits_{|l|\leq N_\tau/2}\widehat{U}_\ell(t)\fe^{i\ell\tau},$$
 and as a convention,
the first term and the last term in the summation are taken with a factor
$1/2$ here and after.
 We collect all the $\widehat{U}_\ell(t)$ in $\widehat{\mathbf{U}}$ and get a  $D:=2d(N_{\tau}+1)$-dimensional system
 of ODEs
\begin{equation}\label{2scale Fourier CPD2}
\frac{d}{dt}\widehat{\mathbf{U}}=i \Omega \widehat{\mathbf{U}}
+\mathbf{J}f\left(\widehat{\mathbf{U}}\right),
\end{equation}
  with $f\left(\widehat{\mathbf{U}}\right)=\mathcal{F}\left(\partial_{U} H_1\left( \textbf{K}  \mathcal{F}^{-1}\widehat{\mathbf{U}}\right)\right)$, $\mathbf{J}=J\otimes I_{N_{\tau}+1}$, and a diagonal matrix $\Omega=\textmd{diag}(\Omega_1,\ldots,\Omega_{2d})$:
  \begin{equation*}\begin{array}[c]{ll}
&\Omega_1=\Omega_2=\cdots=\Omega_{2d}:=\frac{1}{\eps}\textmd{diag}\big(\frac{N_{\tau}}{2},
  \frac{N_{\tau}}{2}-1,\ldots,-\frac{N_{\tau}}{2}\big).
\end{array}
\end{equation*}
 Here  {$\textbf{K}$ is the matrix obtained from $\fe^{\tau A}$ with the semi-discretiation in $\tau$.}
The  coefficient vector $\widehat{\mathbf{U}}$ is in dimension $D$ and its components are  denoted by
 $$
 \widehat{\mathbf{U}}=\left(\widehat{U}_{-\frac{N_{\tau}}{2},1},\ldots,
 \widehat{U}_{\frac{N_{\tau}}{2},1},\widehat{U}_{-\frac{N_{\tau}}{2},2},\ldots,
 \widehat{U}_{\frac{N_{\tau}}{2},2},\ldots,\widehat{U}_{-\frac{N_{\tau}}{2},2d},\ldots,
 \widehat{U}_{\frac{N_{\tau}}{2},2d}\right).$$
With these notations, the components of the nonlinearity are of the form
  \begin{equation}\label{potential}f_{j,l}\left(\widehat{\mathbf{U}}\right)=\frac{\partial}{\partial \widehat{\mathbf{U}}_{-j,l}}V\left(\widehat{\mathbf{U}}\right)\ \ \textmd{with}\ \ V\left(\widehat{\mathbf{U}}\right)=\frac{1}{2d(N_{\tau}+1)}\sum\limits_{l=1}^{2d}
  \sum\limits_{k=-N_{\tau}/2}^{N_{\tau}/2}H_1\left(\left( \textbf{K} \mathcal{F}^{-1}\widehat{\mathbf{U}}\right)_{k,l}\right),\end{equation}
and we denote $\mathcal{F} \widehat{\mathbf{U}}$
the discrete Fourier transform on each $\widehat{\mathbf{U}}_{:,l}$ where $$\widehat{\mathbf{U}}_{:,l}=\left(\widehat{U}_{-\frac{N_{\tau}}{2},l},\widehat{U}_{-\frac{N_{\tau}}{2}+1,l},\ldots,
 \widehat{U}_{\frac{N_{\tau}}{2},l}\right).$$
Similar   notations will be used  for the
vectors and diagonal matrices with the same dimension as $\widehat{\mathbf{U}}$.

In the main theorem, we will need a non-resonance condition and to describe it, we further introduce the following  notations. Let
\begin{equation*}\displaystyle
\begin{array}[c]{ll}
&
\mathbf{k}=\left(k_{-\frac{N_{\tau}}{2},1},\ldots,
k_{\frac{N_{\tau}}{2},1},k_{-\frac{N_{\tau}}{2},2},\ldots,
 k_{\frac{N_{\tau}}{2},2},\ldots,k_{-\frac{N_{\tau}}{2},2d},\ldots,
k_{\frac{N_{\tau}}{2},2d}\right),\\
&\displaystyle |\mathbf{k}|=\sum_{l=1}^{2d}\sum_{j=-N_{\tau}/2}^{N_{\tau}/2}
|k_{j,l}|,\ \
 \boldsymbol{\omega}=(\textmd{diagonal elements of}\  \Omega),\ \ \mathbf{k}\cdot \boldsymbol{\omega}=\sum_{l=1}^{2d}\sum_{j=-N_{\tau}/2}^{N_{\tau}/2}
k_{j,l}\omega_{j,l},
\end{array}
\end{equation*}
and the resonance module
$\mathcal{M}= \{\mathbf{k}\in \mathbb{Z}^{D}:\ \mathbf{k}\cdot
\boldsymbol{\omega}=0\}$.
Denote by $\mathcal{K}$   a set of representatives of the
equivalence classes in $\mathbb{Z}^{D}\backslash \mathcal{M}$, which
are chosen such that for each $\mathbf{k}\in\mathcal{K}$ the sum $|\mathbf{k}|$ is the
minimal in the equivalence class $[\mathbf{k}] = \mathbf{k} +\mathcal{M}$ and also $-\mathbf{k}\in\mathcal{K}.$  For an integer $N>0$, we denote
\begin{equation*}\begin{array}{ll}\mathcal{N}&=\{\mathbf{k}\in\mathcal{K}:\ \textmd{there exists an}\ l\in \{1,\ldots,2d\}\
 \textmd{such that}\
|\mathbf{k}_{:,l}|=|\mathbf{k}| \textmd{ and }|\mathbf{k}|\leq N\},\\
\mathcal{N}^*&=\mathcal{N} \bigcup
\{\langle 0\rangle_l\}_{l=1,2,\ldots,2d},
\end{array} \end{equation*}
where $\langle j\rangle_l$ denotes the unit coordinate vector $(0, \ldots , 0, 1, 0,
\ldots,0)^{\intercal}\in \mathbb{R}^{D}$  with the only entry $1$ at the $(j,l)$-th
position. $\norm{\cdot}$ in this section denotes the vector $2$-norm or the matrix $2$-norm.

The detailed long-term analysis in the following will be made on the \textbf{SE1-TSI} method:
    \begin{equation}\begin{array}[c]{ll} &\widehat{\mathbf{U}}^{n+\frac{1}{2}}=\fe^{hi\Omega/2 }\widehat{\mathbf{U}}^n+\frac{1}{2}h\varphi_1 (hi\Omega/2 )
 g\left(\widehat{\mathbf{U}}^{n+\frac{1}{2}}\right),\ \
\widehat{\mathbf{U}}^{n+1}=\fe^{hi \Omega}\widehat{\mathbf{U}}^n+h
\varphi_1 (hi \Omega)g\left(\widehat{\mathbf{U}}^{n+\frac{1}{2}}\right),
\end{array} \label{SeM1}%
\end{equation}
for solving \eqref{2scale Fourier CPD2} with  $g(\widehat{\mathbf{U}})=\mathbf{J}f(\widehat{\mathbf{U}}).$
The same can be done for SE2-TSI with little modifications in the proof part, which will be omitted for simplicity in the section.

\subsection{Main result}
The near-conservation laws of the proposed methods are stated as follows.
\begin{theorem} \label{Long-time thm} 
\emph{(Near-conservation laws).}
Assume that the initial value $u_0$ of \eqref{model}  {satisfies} $0<\delta_0:=\norm{u_0}<1$ and the nonlinearity $F(u):=\nabla H_1(u)$ is smooth with $F(0)=F'(0)=0.$
Under the step size requirement $h/\sqrt{\eps} \geq c_0 > 0$ and
the numerical non-resonance condition
\begin{equation}
\left|\sin\left(\frac{1}{2}h(\omega_{j,1}- \mathbf{k}\cdot \boldsymbol{\omega})\right)\right| \geq c_1 \sqrt{\eps}\ \
\mathrm{for} \ \ j=-\frac{N_{\tau}}{2},\ldots,\frac{N_{\tau}}{2}, \ \ \mathbf{k} \in \mathcal{N}
   \ \   \mathrm{with} \ \  |\mathbf{k}|\leq N,\label{numerical non-resonance cond}%
\end{equation}
for some large $N\geq3$ and $c_1>0$,
we have the following long-time near-conservation laws for the $u^n$ produced by SE1-TSI \eqref{S1} or SE2-TSI \eqref{S2},
\begin{equation*}
\begin{aligned}
 \frac{\eps}{\delta_0^2} \big(H\left(u^n\right)- H\left(u^0\right)\big)&={\mathcal{O}\left(\eps\delta_0^2\right)}+\mathcal{O}(  \delta_{\mathcal{F}}),\quad
 \frac{\eps}{\delta_0^2} \big(I\left(u^n\right)- I\left(u^0\right)\big)={\mathcal{O}\left(\eps \delta_0^2\right)}+\mathcal{O}(  \delta_{\mathcal{F}}),\\
 \frac{1}{\delta_0^2} \big(m\left(u^n\right) -m\left(u^0\right)\big)&={\mathcal{O}\left(\eps \delta_0^2\right)}+\mathcal{O}(  \delta_{\mathcal{F}}),\quad 0\leq nh\leq \frac{1}{h}\delta_0^{-N+3}.
\end{aligned}
\end{equation*}
Here we denote $\delta_{\mathcal{F}}$  the error brought by the Fourier pseudospectral method. The constants symbolized by
$\mathcal{O}$ depend on $N, c_0, c_1$ but not on $n,\ h,\ \eps$.
\end{theorem}

Note that since  $N$ can be arbitrarily large, {the near-conservation} results hold for a long time.
In practice, as $N_{\tau}$ increases the $\mathcal{O}( \delta_{\mathcal{F}})$ part of the error quickly decreases to a high-precision thanks to the periodicity and the smoothness in $\tau$,
and so the main part of the (relative) conservation error is $\mathcal{O}(\eps\delta_0^2)$.
The proof of the theorem for SE1-TSI will be given in Sections \ref{subsec2}-\ref{subsec3} by using the modulated Fourier expansion \cite{Lubich ICM,Lubich}. Before we step into the proof, we have
some remarks regarding the assumptions in the theorem.

\begin{remark}
The parameter $\delta_0$ denotes the scale of the initial data of the original system, and its smallness, i.e., $\delta_0\in(0,1)$ in Theorem \ref{Long-time thm} is technically required often in the long-time analysis. However, we should note here that $\delta_0$  is totally independent from $\eps$. The requirement on $h$ and $\eps$ comes from the technique of the modulated Fourier expansion that has been introduced to analyze numerical integrators particularly
with comparatively large time step (see, e.g., \cite{Cohen0,Gauckler,Hairer,Lubich,LubichW,WZ}), and this indeed matches with our purpose of designing UA schemes.
\end{remark}
\begin{remark}The assumption $F(0)=0$ is essential for our analysis, and various Hamiltonian nonlinear models indeed satisfy such assumption.
$F'(0)=0$ is however not essential, since the linearization of $F(u)$ can be moved to the linear part of  \eqref{model}. This procedure gives a new nonlinearity $\widetilde{F}(u)=F(u)-F'(0)u$ and it satisfies $\widetilde{F}'(0)=0$.
\end{remark}
\begin{remark}
The result of Theorem \ref{Long-time thm} is established for a general $F(u)$.
For a high order pure power nonlinearity, e.g., our numerical example on the NLS equation (\ref{NLS eq}) with a cubic nonlinearity, improved long-time
near-conservation laws could be derived. This will be discussed in Section \ref{subsec-NS}.
\end{remark}

\subsection{Modulated Fourier expansion}\label{subsec2}
In this subsection, we derive the modulated Fourier expansion of the
SE1-TSI integrator.

\begin{lemma}\label{MFE lem}
 Under the  conditions given in Theorem \ref{Long-time thm} and  for $ t=nh$, the
SE1-TSI method \eqref{SeM1}  can be expressed by the
following modulated Fourier expansion,
\begin{equation}\begin{array}{ll}
\widehat{\mathbf{U}}^{n}&=\sum\limits_{\mathbf{k}\in\mathcal{N}^*}
\mathrm{e}^{i \mathbf{k} \cdot \boldsymbol{\omega}
t}\zeta^{\mathbf{k}}(t)+\widetilde{R}_{1,N}(t),\quad
\widehat{\mathbf{U}}^{n+\frac{1}{2}}=\sum\limits_{\mathbf{k}\in\mathcal{N}^*}
\mathrm{e}^{i \mathbf{k} \cdot \boldsymbol{\omega}
t}\eta^{\mathbf{k}}(t)+\widetilde{R}_{2,N}(t),
\end{array} \label{MFE-ERKN-0}%
\end{equation}
 where the remainders  are bounded by
\begin{equation}
 \widetilde{R}_{1,N}(t)=\mathcal{O}\left(t h \delta_0^{N+1}\right),\quad
 \widetilde{R}_{2,N}(t)=\mathcal{O}\left(t h \delta_0^{N+1}\right).
\label{remainder}%
\end{equation}
It turns out that we can assume that $\zeta^{\mathbf{k}}_{0,l}\ \textmd{or}\ \eta^{\mathbf{k}}_{0,l}=0$ if $\mathbf{k}\neq \langle 0\rangle_l$, and the coefficient functions are  single waves,
i.e., $\zeta^{\mathbf{k}}_{:,m}\ \textmd{or}\ \eta^{\mathbf{k}}_{:,m}=0$ if $\abs{\mathbf{k}_{:,l}}>0\ \textmd{and}\ l\neq m$.
The coefficient functions as well as all their derivatives are
bounded by
\begin{equation}%
\begin{array}{ll}
 \eta_{j,l}^{\langle j\rangle_l}(t)=\mathcal{O}(\delta_0),\quad
  \dot{\eta}_{j,l}^{\langle j\rangle_l}(t)=\mathcal{O}\left(\frac{\eps^2}{h}\nu \delta_0^3\right),
 \quad
 \eta_{j,l}^{\mathbf{k}}(t)=\mathcal{O}\Big(\eps \nu \delta_0^{|\mathbf{k}|}\Big),\quad   {\mathbf{k}\neq \langle j\rangle_l,}
\end{array} %
\label{coefficient func}%
\end{equation}
where $l=1,\ldots,2d,$ and $\nu=\frac{1}{c_1 \sqrt{\eps}}$.
 The functions $\zeta^{\mathbf{k}}$ satisfy the same bounds of $\eta^{\mathbf{k}}$, and $\zeta_{j,l}^{\langle j\rangle_l}$ have the following relationship with $\eta_{j,l}^{\langle j\rangle_l}$:
\begin{equation}%
\begin{array}{ll}
 \zeta_{j,l}^{\langle j\rangle_l}(t)=\fe^{ih\omega_{j,l}/2 } \eta_{j,l}^{\langle j\rangle_l} (t)+\mathcal{O}(h) \dot{\eta}_{j,l}^{\langle j\rangle_l} (t)
 =\fe^{ih\omega_{j,l}/2 } \eta_{j,l}^{\langle j\rangle_l} (t)+\mathcal{O}(\eps^2\nu \delta_0^3).
\end{array} %
\label{coefficient func2}%
\end{equation}
 Since $\widetilde{\mathbf{U}}^{n} \in \mathbb{R}^{2d}\approx \sum\limits_{\mathbf{k}\in\mathcal{N}^*}
\mathrm{e}^{i \mathbf{k} \cdot \boldsymbol{\omega}
t}  \sum\limits_{|l|\leq N_\tau/2}\zeta^{\mathbf{k}}_{l,j}\fe^{i l\tau}$,
it is required that $\zeta^{-\mathbf{k}}_{-l,j}=\overline{\zeta^{\mathbf{k}}_{l,j}}$. Similar results also hold for $\eta^{\mathbf{k}}$.
From the bounds of the coefficients, it immediately  follows that
\begin{equation}\begin{aligned}
&\mathbf{r^1}_{:,l}:=\sum\limits_{\mathbf{k}\in\mathcal{N}^*}
\mathrm{e}^{i \mathbf{k} \cdot \boldsymbol{\omega}
t}\zeta_{:,l}^{\mathbf{k}}(t)-\zeta_{:,l}^{\langle j\rangle_l}(t)
\mathrm{e}^{i  \omega_{j,l} t}-\zeta_{:,l}^{-\langle j\rangle_l}(t)
\mathrm{e}^{-i  \omega_{j,l} t}  \quad
\textmd{with} \quad \norm{\mathbf{r^1}}\leq C \eps \nu \delta_0^2,\\
&\mathbf{r^2}_{:,l}:=\sum\limits_{\mathbf{k}\in\mathcal{N}^*}
\mathrm{e}^{i \mathbf{k} \cdot \boldsymbol{\omega}
t}\eta_{:,l}^{\mathbf{k}}(t)-\eta_{:,l}^{\langle j\rangle_l}(t)
\mathrm{e}^{i  \omega_{j,l} t}  -\eta_{:,l}^{-\langle j\rangle_l}(t)
\mathrm{e}^{-i  \omega_{j,l} t}\quad
\textmd{with} \quad \norm{\mathbf{r^2}}\leq C \eps \nu \delta_0^2.
\end{aligned}
\label{jth TMFE}
\end{equation}
The constants symbolised by the notation $\mathcal{O}$ depend on $ c_0, c_1$, but are independent of $h, \eps$ or $N_{\tau}$.
\end{lemma}
\begin{proof}
In the proof of this lemma, we will construct the functions
\begin{equation}
\begin{array}{ll}
\Phi(t)&= \sum\limits_{\mathbf{k}\in\mathcal{N}^*}
\mathrm{e}^{i \mathbf{k} \cdot\boldsymbol{\omega}
t}\zeta^{\mathbf{k}}(t),\quad
\Psi(t)= \sum\limits_{\mathbf{k}\in\mathcal{N}^*}
\mathrm{e}^{i \mathbf{k} \cdot\boldsymbol{\omega}
t}\eta^{\mathbf{k}}(t),
\end{array} \label{MFE-1}%
\end{equation} with smooth coefficient functions $\zeta^{\mathbf{k}}$ and $\eta^{\mathbf{k}}$, and show   that there are
only   small defects $ \widetilde{R}_{1,N}$ and $ \widetilde{R}_{2,N}$  when $\Phi(t), \Psi(t)$ are inserted into the
numerical scheme \eqref{SeM1}.

From the scheme \eqref{SeM1}, it follows that
    \begin{equation*} \widehat{\mathbf{U}}^{n+\frac{1}{2}}=\fe^{  hi\Omega/2  }\widehat{\mathbf{U}}^n+\frac{1}{2}\varphi_1 (hi\Omega/2)\varphi^{-1}_1 (hi \Omega )
\left(\widehat{\mathbf{U}}^{n+1}-\fe^{hi \Omega }\widehat{\mathbf{U}}^n\right).
\end{equation*}
Inserting  \eqref{MFE-1}   into this one and
using the operator  \begin{equation*}\mathcal{L}_1(h\mathcal{D})= \fe^{hi\Omega/2  }+\frac{1}{2}\varphi_1 (hi\Omega/2 )\varphi^{-1}_1 (hi \Omega )\left(\fe^{h\mathcal{D}} -\fe^{hi \Omega } \right),
\end{equation*}
one finds
 \begin{equation}\label{etar}
 \eta^{\mathbf{k}}(t)=\mathcal{L}_1(h\mathcal{D}+i h\mathbf{k} \cdot\boldsymbol{\omega})\zeta^{\mathbf{k}}(t).
\end{equation}
Here  $\mathcal{D}$ is the differential operator (see \cite{Lubich}).
With \eqref{MFE-1} and the second equation in \eqref{SeM1}, we get
 \begin{equation*}
 \mathcal{L}(h\mathcal{D}) \Psi(t)=hg(\Psi(t)),
\end{equation*}
where
 \begin{equation*}
 \mathcal{L}(h\mathcal{D})= \big(\fe^{h\mathcal{D}} -\fe^{hi \Omega } \big)\varphi^{-1}_1 (hi \Omega ) \mathcal{L}^{-1}_1(h\mathcal{D}).
\end{equation*}
Considering the Taylor series of the
nonlinearity, one obtains
\begin{equation*}
\begin{aligned}&\mathcal{L}(h\mathcal{D})\Psi(t)
=h\sum\limits_{\mathbf{k}\in\mathcal{N}^*}\mathrm{e}^{i \mathbf{k}
\cdot \boldsymbol{\omega}t} \sum\limits_{m\geq
2}\frac{g^{(m)}(0)}{m!}
\sum\limits_{\mathbf{k}^1+\cdots+\mathbf{k}^m=\mathbf{k}} \Big[\eta^{\mathbf{k}^1}\cdot\ldots\cdot
\eta^{\mathbf{k}^m}\Big](t).
\end{aligned} %
\end{equation*}
Inserting  the ansatz \eqref{MFE-1} and then comparing the coefficients
of $\mathrm{e}^{i (\mathbf{k} \cdot\boldsymbol{\omega} ) t}$  yield
\begin{equation*}
\begin{aligned}
&\mathcal{L}(h\mathcal{D}+ih \mathbf{k} \cdot\boldsymbol{\omega}
)\eta^{\mathbf{k}}(t)=h\sum\limits_{m\geq
2}\frac{g^{(m)}(0)}{m!}
\sum\limits_{\mathbf{k}^1+\cdots+\mathbf{k}^m=\mathbf{k}} \Big[\eta^{\mathbf{k}^1}\cdot\ldots\cdot
\eta^{\mathbf{k}^m}\Big](t),
\end{aligned} %
\end{equation*}
which gives the  modulation system  for the coefficients
$\eta^{\mathbf{k}}(t)$ in the modulated Fourier expansion.

According
to  the Taylor expansion of $  L(h\mathcal{D})$,  one has
\begin{equation}\label{LhD}
\begin{aligned}L(h\mathcal{D})= &-h \Omega i  +\frac{h\Omega/2 }{\sin(h\Omega/2 )}
(h\mathcal{D})+ \frac{1}{8} h \Omega \sec^2(h\Omega/4 )i (h\mathcal{D})^2-\cdots,\\
L(h\mathcal{D}+i h\mathbf{k}\cdot \boldsymbol{\omega})=&-\frac12 h \Omega  \csc(h\Omega/4 )
\sec\left(\frac{1}{4}h(\Omega -2 \mathbf{k}\cdot \boldsymbol{\omega}I)\right)\sin\left(\frac{1}{2}h(\Omega -\mathbf{k}\cdot \boldsymbol{\omega}I)\right)
i \\
&+\frac14 h  \Omega \cot(h\Omega/4 ) \sec^2\left(\frac{1}{4}h(\Omega -2\mathbf{k}\cdot \boldsymbol{\omega}I)\right) (h\mathcal{D})+\cdots,
\end{aligned}
\end{equation}
and for the   particular components, we have
\begin{equation*}
\begin{aligned}
\big(L(h\mathcal{D}+i h\langle j\rangle_l\cdot \boldsymbol{\omega})\big)_{j,l}=&
\frac{h \omega_{j,l}/2}{\sin(h\omega_{j,l}/2)}h\mathcal{D} -i\frac{h\omega_{j,l}}{8} \sec^2(h\omega_{j,l}/4) (h\mathcal{D})^2+\cdots.\\
\end{aligned}
\end{equation*}
Use these results and then the following ansatz of the modulated
Fourier functions $\eta^{\mathbf{k}}(t)$  can be obtained:
\begin{equation}\label{ansatz}%
\begin{array}{ll}
& \dot{\eta}_{j,l}^{\langle j\rangle_l}(t)=\frac{\sin(h\omega_{j,l}/2)}{h \omega_{j,l}/2}
\big(F^{\mathbf{1}}_{j0}(\cdot)+\cdots\big),\quad j=-\frac{N_{\tau}}{2},\ldots,\frac{N_{\tau}}{2},\\
&\eta_{j,l}^{\langle 0 \rangle_l}(t)=\frac{1}{\omega_{j,l}}
i \big(F^{\mathbf{0}}_{j0}(\cdot)+\cdots\big),\quad j\neq0,\\
&\eta_{j,l}^{\mathbf{k}}(t)=\frac{-2\sin(h\Omega/4_{j,l})\cos\big(\frac{1}{4}h(\omega_{j,l}-2 \mathbf{k}\cdot \boldsymbol{\omega})\big)}{ \omega_{j,l}\sin\big(\frac{1}{2}h(\omega_{j,l}- \mathbf{k}\cdot \boldsymbol{\omega})\big)}
i \big(F^{\mathbf{k}}_{j0}(\cdot)+\cdots\big),\quad
\mathbf{k}\neq\langle j \rangle_l,\\
\end{array} %
\end{equation}
where  $F^{\mathbf{k}}$ and so on are formal
series, and the dots  stand for power series in $h$.  Since the
series in the ansatz usually diverge,  we
 truncate them after the $\mathcal{O}(h^{N+3})$ terms.
We  determine the initial values for the differential equations by
considering the condition that \eqref{MFE-ERKN-0} is satisfied
without remainder term for $t = 0$. From
$\Psi(0)=\widehat{\mathbf{U}}^{1/2}=\mathcal{O}(\widehat{\mathbf{U}}^0)$,
it follows that $\widehat{\mathbf{U}}_{j,l}^{1/2}=\eta^{\langle j\rangle_l}_{j,l}(0)+\mathcal{O}(\eps \nu \delta_0^2).$
This gives the initial values for $\eta^{\langle j\rangle_l}_{j,l}(0)=\mathcal{O}(\delta_0)$.
 The bound \eqref{coefficient func} is obtained on the basis of the initial value and the ansatz \eqref{ansatz}.  In combination with \eqref{etar}, the relationship \eqref{coefficient func2} is derived immediately.

As the last part of the proof, we analyze the  defects
\eqref{remainder}. Firstly, define the discrepancies  when $\Phi(t), \Psi(t)$ are inserted into the
numerical scheme \eqref{SeM1}:
\begin{equation*}
d_1(t)=\Psi(t)-\fe^{hi\Omega/2 }\Phi(t)-\frac{1}{2}h\varphi_1 (hi\Omega/2 )
 g(\Psi(t)),\quad
d_2(t)=\Phi(t+h)-\fe^{hi \Omega}\Phi(t)-h
\varphi_1 (hi \Omega)g(\Psi(t)).
\end{equation*}
These discrepancies come from two aspects:  $\mathcal{O}(h^{N+3})$ in  the truncation of  the ansatz \eqref{ansatz} and
$\mathcal{O}(\eps \nu \delta_0^{N+1})$ in the truncation of the modulated Fourier expansions. This implies
$$d_j(t)=\mathcal{O}\left(h^{N+3}\right)+\mathcal{O}\left(h \eps \nu \delta_0^{N+1}\right)=\mathcal{O}\left(h^2 \delta_0^{N+1}\right)\ \ \textmd{for}\ \ j=1,2.$$
Then,  the  error  $e_n=\widehat{\mathbf{U}}^{n}-\Phi(t_n),\ E_{n+\frac{1}{2}}=\widehat{\mathbf{U}}^{n+\frac{1}{2}}-\Psi(t_n)$ satisfy
\begin{equation*}
\begin{array}[c]{ll} E_{n+\frac{1}{2}}=\fe^{hi\Omega/2}e_{n}+\frac{1}{2}h
\varphi_1 (hi\Omega/2)\left[g\left(\widehat{\mathbf{U}}^{n+\frac{1}{2}}\right)-g\left(\Psi(t_n)\right)\right]
+d_1(t_n),\\
e_{n+1}=\fe^{hi \Omega}e_{n}+h
\varphi_1 (hi \Omega)\left[g\left(\widehat{\mathbf{U}}^{n+\frac{1}{2}}\right)-g(\Psi(t_n))\right]+d_2(t_n).
\end{array}
\end{equation*}
 With the Lipschitz condition of $g$ and Gronwall's inequality, the defects
\eqref{remainder}  can be derived by the standard convergence estimates (see, e.g.,  \cite{Lubich1,Lubich2,Hairer,Lubich,WW}), and the proof is complete.
\end{proof}

\subsection{Proof of  Theorem \ref{Long-time thm}}\label{subsec3}
 \begin{proof}
 In this subsection, we  derive  three almost-invariants  of the
functions in the modulated Fourier expansions which lead to the near-conservation laws in the theorem.

\textbf{Three almost-invariants.}
Let $\vec{\zeta}=\big(\zeta^{\mathbf{k}}\big)_{\mathbf{k}\in \mathcal{N}^*}$ and
$\vec{\eta}=\big(\eta^{\mathbf{k}}\big)_{\mathbf{k}\in \mathcal{N}^*}.$
From the proof of Lemma \ref{MFE lem}, it follows that
\begin{equation*}
\begin{aligned}
& L(h\mathcal{D}) \mathbf{U}_{h}(t)=h
g(\mathbf{U}_{h}(t))+\mathcal{O}\left(h^2\delta_0^{N+1}\right),
\end{aligned}
\end{equation*}
where we use  the   notations \begin{equation*}
\begin{aligned}\mathbf{U}_{h}(t)=\sum\limits_{\mathbf{k}\in\mathcal{N}^*}\mathbf{U}^{\mathbf{k}}_{h}(t)\quad \textmd{with}\quad  \ \mathbf{U}^{\mathbf{k}}_{h}(t)=\mathrm{e}^{i (\mathbf{k} \cdot \boldsymbol{\omega})
t}\eta^{\mathbf{k}}(t).
\end{aligned}
\end{equation*}
Then we have
\begin{equation*}
\begin{aligned} & L(h\mathcal{D}) \mathbf{U}^{\mathbf{k}}_{h}(t)=h\mathbf{J}\nabla_{\mathbf{U}^{-\mathbf{k}}}
\mathcal{V}\left(\vec{\mathbf{U}}(t)\right)
+\mathcal{O}\left(h^2\delta_0^{N+1}\right),
\end{aligned}
\end{equation*}
where $\mathcal{V}(\vec{\mathbf{U}}(t))$ is defined as
\begin{equation}
\begin{aligned}
&\mathcal{V}\left(\vec{\mathbf{U}}(t)\right):=\sum\limits_{m= 1}^N
\frac{V^{(m+1)}(0)}{(m+1)!}
\sum\limits_{\mathbf{k}^1+\cdots+\mathbf{k}^{m+1}=0}
 \left(\mathbf{U}^{\mathbf{k}^1}_{h}\cdot\ldots\cdot
\mathbf{U}^{\mathbf{k}^{m+1}}_{h}\right)(t),
\end{aligned}
\label{newuu}%
\end{equation}
with the potential $V$ given in \eqref{potential} and the notaiton $$\vec{\mathbf{U}}(t)=\left( \mathbf{U}^{\mathbf{k}}_{h}(t) \right)_{\mathbf{k}\in \mathcal{N}^*}.$$
By switching to the quantity $\eta^{\mathbf{k}}(t)$, we obtain \begin{equation}
\begin{aligned} & L(h\mathcal{D}+i h\mathbf{k}
\cdot\boldsymbol{\omega}) \eta^{\mathbf{k}} (t) =h\mathbf{J}\nabla_{\eta^{-\mathbf{k}}}\mathcal{V}(\vec{\eta}(t))
+\mathcal{O}\left(h^2\delta_0^{N+1}\right).
\end{aligned}\label{lhdy}%
\end{equation}

 Define
the vector functions  of $\vec{\vartheta}(\lambda,t)$ as
$$\vec{\vartheta}(\lambda,t)=\left(\mathrm{e}^{i (\mathbf{k} \cdot
\boldsymbol{\mu}) \lambda}\eta^{\mathbf{k}}(t)\right)_{\mathbf{k}\in\mathcal{N}^*},$$
for any real sequence $\boldsymbol{\mu}$.
 Then it can be observed from the definition \eqref{newuu}
  that $\mathbf{k}^1+\cdots+\mathbf{k}^{m+1}=0$. Thus,
  $\mathcal{V} (\vec{\vartheta}(\lambda,t))$ does  not depend on
$\lambda$, and then the following result is obtained
\begin{equation*}
\begin{aligned}0=& \left.\frac{d}{d\lambda}\right|_{\lambda=0}\mathcal{V}
\left(\vec{\vartheta}(\lambda,t)\right)=\sum\limits_{\mathbf{k}\in\mathcal{N}^*}i (\mathbf{k}
\cdot\boldsymbol{\mu} )\left(\eta^{-\mathbf{k}}(t)\right)^\intercal
\nabla_{{\eta^{-\mathbf{k}}}}\mathcal{V} (\vec{\eta}(t)).
\end{aligned}
\end{equation*}
In addition to \eqref{lhdy}, we also obtain
\begin{equation}\label{almost in 2}%
\begin{aligned}0
=&\frac{1}{h}\sum\limits_{\mathbf{k}\in\mathcal{N}^*} i (\mathbf{k}
\cdot\boldsymbol{\mu} ) (\eta^{-\mathbf{k}}(t))^\intercal  \mathbf{J}L(h\mathcal{D}+i h\mathbf{k}
\cdot \boldsymbol{\omega} ) \eta^{\mathbf{k}}(t)+\mathcal{O}\left(h\delta_0^{N+1}\right).
\end{aligned}
\end{equation}
By the relation $\eta^{-\mathbf{k}}_{-l,j}=\overline{\eta^{\mathbf{k}}_{l,j}}$, we introduce a new expression
$\overline{\eta^{-\mathbf{k}}}=\textbf{S}\eta^{\mathbf{k}}$, where
  \begin{equation*}
\textbf{S}=\textmd{diag}(S_1,S_2,\ldots,S_{2d})\quad \mbox{with}\begin{array}[c]{ll}
&S_1=S_2=\cdots=S_{2d}:=\left(
                       \begin{array}{cccc}
                         0 & \cdots & 0 &1 \\
                         0 & \cdots & 1 & 0 \\
                          \vdots &\vdots & \vdots & \vdots \\
                         1& \cdots & 0 & 0 \\
                       \end{array}
                     \right)_{(N_{\tau}+1)\times (N_{\tau}+1)}.
\end{array}
\end{equation*}
Consequently,  \eqref{almost in 2} becomes \begin{equation*}
\begin{aligned}0
=&\frac{1}{h}\sum\limits_{\mathbf{k}\in\mathcal{N}^*} i (\mathbf{k}
\cdot\boldsymbol{\mu} )\left(\overline{\eta^{\mathbf{k}}(t)}\right)^\intercal  \textbf{S}\mathbf{J}L(h\mathcal{D}+i h\mathbf{k}
\cdot \boldsymbol{\omega}) \eta^{\mathbf{k}}(t)+\mathcal{O}\left(h\delta_0^{N+1} \right).
\end{aligned}
\end{equation*}
It is easy to check that $ \textbf{S}\mathbf{J}= \mathbf{J}\textbf{S}$, which gives the skew-symmetry of $ \textbf{S}\mathbf{J}$. Therefore, there exist  a unitary matrix $\textbf{P}$
and a diagonal matrix $\Lambda=\textmd{diag}(-I,I)$ such that $ \textbf{S}\textbf{J}=  i \textbf{P}^\textup{H}
\Lambda \textbf{P}.$
We now have \begin{equation*}
\begin{aligned}\frac{1}{h} \sum\limits_{\mathbf{k}\in\mathcal{N}^*} i (\mathbf{k}
\cdot\boldsymbol{\mu} ) (\overline{\eta^{\mathbf{k}}(t)})^\intercal   i \textbf{P}^\textup{H}
\Lambda \textbf{P}L(h\mathcal{D}+i h\mathbf{k}
\cdot\boldsymbol{\omega}) \eta^{\mathbf{k}}(t)=\mathcal{O}\left(h\delta_0^{N+1}\right).
\end{aligned}
\end{equation*}
In the light of the schemes of  $\textbf{P}$ and $\eta^{\mathbf{k}}(t)$, 
one can verify that {the result $\textbf{P}L(h\mathcal{D}+i h\mathbf{k}
\cdot\boldsymbol{\omega}) \eta^{\mathbf{k}}(t)$ has two possible expressions:
$L(h\mathcal{D}+i h\mathbf{k}
\cdot\boldsymbol{\omega})\textbf{P}\eta^{\mathbf{k}}(t)$ or $\widetilde{L}(h\mathcal{D}+i h\mathbf{k}
\cdot\boldsymbol{\omega})\textbf{P}\eta^{\mathbf{k}}(t)$,
where $\widetilde{L}_{j,l}=L_{-j,l}$ for $j=-N_{\tau}/2,\ldots, N_{\tau}/2$ and $l=1,\ldots,2d.$ In what follows, we only prove the result for the second case and the proof also holds for the first case.}
This yields the result
\begin{equation}\label{duu-n1}
\begin{aligned}\frac{1}{h} \sum\limits_{\mathbf{k}\in\mathcal{N}^*}  (\mathbf{k}
\cdot\boldsymbol{\mu} ) (\overline{\alpha^{\mathbf{k}}  })^\intercal
 { \Lambda \widetilde{L}(h\mathcal{D}+i h\mathbf{k}
\cdot \boldsymbol{\omega} )} \alpha^{\mathbf{k}}=\mathcal{O}\left(h\delta_0^{N+1}\right),
\end{aligned}
\end{equation}
with $\alpha^{\mathbf{k}}=(\textbf{P}\eta^{\mathbf{k}})$.

Thanks to the Taylor expansions of $ L(h\mathcal{D})$   given in
 \eqref{LhD} and the following formulae
 \cite{Lubich} {for a column vector $y=y(t)$ and its $l$-th derivative $y^{(l)}(t)$ for any $l\in\bN$:}
\begin{equation*}
\begin{aligned}
&\textmd{Re}\left(\bar{y}^\intercal y^{(2l+1)}\right)=\textmd{Re}
\frac{\mathrm{d}}{\mathrm{d}t}\left(\bar{y}^\intercal y^{(2l)}-\cdots\pm
\left(\bar{y}^{(l-1)}\right)^\intercal y^{(l+1)}\mp \frac{1}{2}\left(\bar{y}^{(l)}\right)^\intercal y^{(l)}\right),\\
&\textmd{Im}\left(\bar{y}^\intercal y^{(2l+2)}\right)=\textmd{Im}
\frac{\mathrm{d}}{\mathrm{d}t}\left(\bar{y}^\intercal
y^{(2l+1)}-\dot{\bar{y}}^\intercal y^{(2l)}+\cdots\pm
\left(\bar{y}^{(l)}\right)^\intercal y^{(l+1)}\right),
\end{aligned}
\end{equation*}
it is easy to check that   the real part in the
 left-hand side of
\eqref{duu-n1} is a total derivative.  Hence, there exists a function
${\mathcal{I}(t)}$ such
 that
$\frac{\mathrm{d}}{\mathrm{d}t}{\mathcal{I}(t)}=\mathcal{O}(h\delta_0^{N+1})$
and the statement
\begin{equation*}
{\mathcal{I}(t)=\mathcal{I}(0)}+\mathcal{O}\left(th\delta_0^{N+1}\right)
\end{equation*}
 is obtained by an integration.
Concrete construction of $\mathcal{I}$ can be obtained by choosing some $\boldsymbol{\mu}$,
which is shown as follows.

With the previous expression $ L(h\mathcal{D})$   given in
 \eqref{LhD} and  the bounds
 of Lemma \ref{MFE lem}, we obtain
\begin{equation*}\begin{aligned}
&\mathcal{I}_1(t)=   \sum\limits_{l=1}^{2d} \sum\limits_{j=-N_{\tau}/2}^{N_{\tau}/2} \mu_{j,l}\Big\{
\big(\overline{\alpha^{\langle j \rangle_l}}\big)^\intercal (t)
\mathcal{M}_1\left(\widetilde{\Omega} \right)
\alpha^{\langle j \rangle_l} (t) {-
\big(\overline{\alpha^{\langle j \rangle_l}}\big)^\intercal (t)\mathcal{M}_2\left(\widetilde{\Omega} \right)
 \dot{\alpha}^{\langle j \rangle_l} (t)\Big\}
  +\mathcal{O}\left( \nu^2\eps^2\delta_0^4\right)}\\
 =&   \sum\limits_{l=1}^{2d} \sum\limits_{j=-N_{\tau}/2}^{N_{\tau}/2} \mu_{j,l}\Big\{
\big(\overline{\eta^{\langle j \rangle_l}}\big)^\intercal (t)
\textbf{P}^\textup{H} \mathcal{M}_1\left(\widetilde{\Omega} \right) \textbf{P}\eta^{\langle j \rangle_l} (t)
 {-
\big(\overline{\eta^{\langle j \rangle_l}}\big)^\intercal (t)
\textbf{P}^\textup{H}  \mathcal{M}_2\left(\widetilde{\Omega} \right) \textbf{P}\dot{\eta}^{\langle j \rangle_l} (t)\Big\}{+\mathcal{O}\left( \nu^2\eps^2\delta_0^4\right)},}
\end{aligned}
\end{equation*}
where we use the notations
\begin{equation*}\begin{aligned}
\mathcal{M}_1\left(\widetilde{\Omega}\right)=& \Lambda
  \frac{\cos\big(\frac{1}{4}h\widetilde{\Omega} \big)}{\mathrm{sinc}\big(\frac{1}{4}h\widetilde{\Omega} \big)}
  \sec\left(\frac{1}{4}h\left(2\omega_{j,l}I-\widetilde{\Omega}\right)\right),\\
\mathcal{M}_2\left(\widetilde{\Omega} \right)=&\Lambda\frac{1}{8}h^2\widetilde{\Omega} \cot\left(\frac{1}{4}
h\widetilde{\Omega} \right)\sec^2\left(\frac{1}{4}h\left(2\omega_{j,l}I-\widetilde{\Omega} \right)\right)
\tan^2\left(\frac{1}{4}h\left(2\omega_{j,l}I-\widetilde{\Omega}\right)\right),
\end{aligned}
\end{equation*}
with the matrix $\widetilde{\Omega} _{j,l}=\Omega_{-j,l}$ for $j=-N_{\tau}/2,\ldots, N_{\tau}/2$ and $l=1,\ldots,2d.$
In addition to  \begin{equation*}\begin{aligned}\textbf{P}^\textup{H}
\mathcal{M}_j\left(\widetilde{\Omega}\right) \textbf{P}\eta^{\langle j \rangle_l} (t)=&\textbf{P}^\textup{H} \textbf{P} \mathcal{M}_j(\Omega)\eta^{\langle j \rangle_l} (t)=\mathcal{M}_j(\Omega)\eta^{\langle j \rangle_l} (t),\ \  \textmd{for}\ \ j=1,2,\end{aligned}
\end{equation*} the above  $\mathcal{I}_1(t)$ changes to
\begin{equation*}\begin{aligned}
&\mathcal{I}_1(t)=  \sum\limits_{l=1}^{2d} \sum\limits_{j=-N_{\tau}/2}^{N_{\tau}/2} \mu_{j,l}\Big\{
\big(\overline{\eta^{\langle j \rangle_l}}\big)^\intercal (t)\mathcal{M}_1(\Omega)\eta^{\langle j \rangle_l} (t)
-\big(\overline{\eta^{\langle j \rangle_l}}\big)^\intercal (t)\mathcal{M}_2(\Omega)\dot{\eta}^{\langle j \rangle_l} (t)\Big\}+\mathcal{O}\left(\nu^2\eps^2\delta_0^4\right)\\
 =&  \sum\limits_{l=1}^{2d} \sum\limits_{j=-N_{\tau}/2}^{N_{\tau}/2} (-1)^{\lfloor\frac{2l-1}{2d}\rfloor}\mu_{j,l}\Big\{
 \frac{1}{\mathrm{sinc}\big(\frac{1}{2}h\omega_{j,l}\big)} \abs{\eta^{\langle j \rangle_l}_{j,l}}^2(t)-
 \frac{h^2\omega_{j,l}}{8}\sec^2\left(\frac{1}{4}h\omega_{j,l} \right)
\textmd{Re} \left(\eta^{\langle j \rangle_l}_{j,l} \dot{\eta}^{\langle j \rangle_l}_{j,l}
 \right)(t)\Big\}\\&+\mathcal{O}\left(\nu^2\eps^2\delta_0^4\right).
\end{aligned}
\end{equation*}
 Considering  $\mu_{j,l}=(-1)^{\lfloor\frac{2l-1}{2d}\rfloor}M(l,l)\mathrm{sinc}\big(\frac{1}{2}h\omega_{j,l}\big)$ leads to
\begin{equation*}\begin{aligned}
&\mathcal{I}_1(t)\\
=& \sum\limits_{l=1}^{2d} \sum\limits_{j=-N_{\tau}/2}^{N_{\tau}/2} M(l,l)
  \Big\{\abs{\eta^{\langle j \rangle_l}_{j,l}}^2(t) -
 \frac{h}{4}\sin\left(\frac{1}{2}h\omega_{j,l} \right) \sec^2\left(\frac{1}{4}h\omega_{j,l} \right)
\textmd{Re} \left(\eta^{\langle j \rangle_l}_{j,l} \dot{\eta}^{\langle j \rangle_l}_{j,l}
 \right)(t)\Big\}+\mathcal{O}(\nu^2\eps^2\delta_0^4)\\
=& \sum\limits_{l=1}^{2d} \sum\limits_{j=-N_{\tau}/2}^{N_{\tau}/2} M(l,l)
  \abs{\zeta^{\langle j \rangle_l}_{j,l}}^2(t)+\mathcal{O}\left(\nu^2\eps^2\delta_0^4\right).
\end{aligned}
\end{equation*}
Here we have used the bounds given in  \eqref{coefficient func}  and we further get
$$\frac{h}{4}\sin\left(\frac{1}{2}h\omega_{j,l} \right) \sec^2\left(\frac{1}{4}h\omega_{j,l} \right)
\textmd{Re} \left(\eta^{\langle j \rangle_l}_{j,l} \dot{\eta}^{\langle j \rangle_l}_{j,l}
 \right)(t)=\mathcal{O}\left( \nu\eps^2\delta_0^4\right).$$ We thus obtain the first almost-invariant $\mathcal{I}_1(t)$ of  the
functions in the modulated Fourier expansions.
Using the same arguments, the second  almost-invariant  can be derived and it has the form \begin{equation*}\begin{aligned}
\mathcal{M}(t)= \sum\limits_{l=1}^{2d} \sum\limits_{j=-N_{\tau}/2}^{N_{\tau}/2}
  \abs{\zeta^{\langle j \rangle_l}_{j,l}}^2(t)+\mathcal{O}\left(\nu^2\eps^2\delta_0^4\right).
\end{aligned}
\end{equation*}

In what follows, we will derive the third almost-invariant.
 Multiplying \eqref{lhdy} with $ \big( \dot{\eta}^{-\mathbf{k}}\big)^\intercal$ and summing up implies
\begin{equation}\label{A2}
\begin{aligned} \mathcal{O}\left(h\delta_0^{N+1}\right)=&\frac{1}{h}\sum\limits_{k\in\mathcal{N}^*}
 \left( \dot{\eta}^{-\mathbf{k}}\right)^\intercal \mathbf{J}
L(h\mathcal{D}+i h\mathbf{k}
\cdot\boldsymbol{\omega}) \eta^{\mathbf{k}}+\frac{\mathrm{d}}{\mathrm{d}t}\mathcal{V}(\vec{\eta}(t))\\
=&\frac{1}{h}\sum\limits_{k\in\mathcal{N}^*}
\left(\overline{\dot{\alpha}^{\mathbf{k}}  }\right)^\intercal i  { \Lambda \widetilde{L}(h\mathcal{D}+i h\mathbf{k}
\cdot \boldsymbol{\omega})} \alpha^{\mathbf{k}}+\frac{\mathrm{d}}{\mathrm{d}t}\mathcal{V}(\vec{\eta}(t)).
\end{aligned}
\end{equation}
Using the same arguments as above and the identities {for a column vector $y(t)$:}
\begin{equation*}
\begin{aligned}
&\textmd{Re}\left(\dot{\bar{y}}^\intercal y^{(2l)}\right)=\textmd{Re}
\frac{\mathrm{d}}{\mathrm{d}t}\left(\dot{\bar{y}}^\intercal y^{(2l-1)}-\cdots\mp
\left(\bar{y}^{(l-1)}\right)^\intercal y^{(l+1)}\pm \frac{1}{2}\left(\bar{y}^{(l)}\right)^\intercal y^{(l)}\right),\\
&\textmd{Im}\left(\dot{\bar{y}}^\intercal y^{(2l+1)}\right)=\textmd{Im}
\frac{\mathrm{d}}{\mathrm{d}t}\left(\dot{\bar{y}}^\intercal y^{(2l)}-
\ddot{\bar{y}}^\intercal y^{(2l-1)}+\cdots\mp \left(\bar{y}^{(l)}\right)^\intercal y^{(l+1)}\right),
\end{aligned}
\end{equation*} it is shown that the real part in the
 right-hand side of
\eqref{A2} is a total derivative.  Therefore, there exists a function
$\mathcal{H}_1$ such
 that
$\frac{\mathrm{d}}{\mathrm{d}t}\mathcal{H}_1[\vec{\eta}](t)=\mathcal{O}(h\delta_0^{N+1})$.
 The construction   of $\mathcal{H}_1$ is shown
by considering the previous analysis and  the bounds of Lemma \ref{MFE lem}:
\begin{equation*}\begin{aligned}
\mathcal{H}_1(t)&=V( \vec{\eta}(t))+{
 \sum\limits_{l=1}^{2d} \sum\limits_{j=-N_{\tau}/2}^{N_{\tau}/2}
 \frac{1}{\mathrm{sinc}\big(\frac{1}{2}h\omega_{j,l}\big)}\left(\overline{\eta^{\langle j \rangle_l}_{j,l}}\right)^\intercal \dot{\eta}^{\langle j \rangle_l}_{j,l}+\mathcal{O}\left( \nu\eps\delta_0^4\right)}=V( \vec{\eta}(t))+{\mathcal{O}\left( \nu\eps\delta_0^4\right)} .
\end{aligned}\end{equation*}
Defining $\mathcal{H}=\frac{1}{\eps}\mathcal{I}_1+\mathcal{H}_1$ gives the third almost-invariant,
and it satisfies
$\eps \mathcal{H}(t)=\eps \mathcal{H}(0)+\mathcal{O}(t h\delta_0^{N+1})$.

\textbf{Near conservations.}
We now turn to the results on $I$ and $H$  for the numerical method. We can find that
\begin{equation*}
\begin{aligned}
 \eps I(u^n)=&\left(\widetilde{\mathbf{U}}^{n}\right)^\intercal M\left(\widetilde{\mathbf{U}}^{n}\right) = \sum\limits_{l=1}^{2d} M(l,l)  \norm{\widetilde{\mathbf{U}}^{n}_{:,l}}^2
  =  \sum\limits_{l=1}^{2d} M(l,l) \norm{\widehat{\mathbf{U}}^{n}_{:,l}}^2+\mathcal{O}(  \delta_{\mathcal{F}})\\
  =&   \sum\limits_{l=1}^{2d} M(l,l)
\sum\limits_{j=-N_{\tau}/2}^{N_{\tau}/2} \norm{\zeta^{\langle
{j}\rangle_l}}^2(t_n)+\mathcal{O}\left( \nu^2 \eps^2 \delta_0^4\right)+\mathcal{O}( \delta_{\mathcal{F}})\\
  =&  \sum\limits_{l=1}^{d} M(l,l)
\sum\limits_{j=-N_{\tau}/2}^{N_{\tau}/2} \abs{\zeta^{\langle
{j}\rangle_l}_{j,l}}^2(t_n)+\mathcal{O}\left( \nu^2 \eps^2 \delta_0^4\right)+\mathcal{O}(\delta_{\mathcal{F}})
=\mathcal{I}_1(t_n)+\mathcal{O}\left( \nu^2 \eps^2 \delta_0^4\right)+\mathcal{O}(\delta_{\mathcal{F}}),
\end{aligned}
\end{equation*}
and similarly,
\begin{equation*}
\begin{aligned}
 \eps H(u^n)&=\left(\widetilde{\mathbf{U}}^{n}\right)^\intercal M\left(\widetilde{\mathbf{U}}^{n}\right)+ \eps H_1\left(\widetilde{\mathbf{U}}^{n}\right) =\mathcal{I}_1(t_n)+ \eps V\left(\widehat{\mathbf{U}}^{n}\right)+\mathcal{O}\left( \nu^2 \eps^2 \delta_0^4\right)+\mathcal{O}(  \delta_{\mathcal{F}})\\
  &= \eps \mathcal{H}(t_n)+\mathcal{O}\left(\nu^2\eps^2\delta_0^4\right)+\mathcal{O}(  \delta_{\mathcal{F}}).
\end{aligned}
\end{equation*}
Meanwhile, the quantity  $m(u^n)$ can be expressed as
\begin{equation*}
\begin{aligned}
 m(u^n) = \mathcal{M}(t_n) +\mathcal{O}\left(\nu^2\eps^2\delta_0^4\right)+\mathcal{O}(\delta_{\mathcal{F}}).
\end{aligned}
\end{equation*}
We now use the above results to get
\begin{align*}
 &\eps H(u^n)=  \eps\mathcal{H}(t_n)+\mathcal{O}\left(\nu^2\eps^2\delta_0^4\right)+\mathcal{O}(  \delta_{\mathcal{F}}) = \eps \mathcal{H}(t_{n-1})+h\mathcal{O}\left(h\delta_0^{N+1}\right)
 +\mathcal{O}\left(\nu^2\eps^2\delta_0^4\right)+\mathcal{O}(  \delta_{\mathcal{F}})
  =\ldots\\
   &= \mathcal{H}(t_{0})+nh\mathcal{O}\left(h\delta_0^{N+1}\right)
   +\mathcal{O}\left(\nu^2\eps^2\delta_0^4\right)+\mathcal{O}(  \delta_{\mathcal{F}}) = \eps \mathcal{H}(t_{0})+nh\mathcal{O}\left(h\delta_0^{N+1}\right)
   +\mathcal{O}\left(\eps\delta_0^4\right)+\mathcal{O}(  \delta_{\mathcal{F}})\\
    &=  \eps H(u^0)+\mathcal{O}\left(\eps\delta_0^4\right)+\mathcal{O}(  \delta_{\mathcal{F}}),\\
\end{align*}
as long as $nh^2\delta_0^{N+1} \leq \eps\delta_0^4$.
Using the identical argument gives the estimates
\begin{equation*}
\begin{aligned}
 \eps I(u^n)=\eps I(u^0)+{\mathcal{O}(\eps\delta_0^4)}+\mathcal{O}(  \delta_{\mathcal{F}}),\quad
  m(u^n)= m(u^0)+{\mathcal{O}\left(\eps\delta_0^4\right)}+\mathcal{O}(  \delta_{\mathcal{F}}),
\end{aligned}
\end{equation*}
 and thus we complete the proof of Theorem \ref{Long-time thm} for SE1-TSI.
\end{proof}

\subsection{Analysis for the NLS model}\label{subsec-NS}
At last, we consider the precise example of the NLS equation (\ref{NLS eq}), where the nonlinearity is cubic rather than quadratic as assumed in  Theorem \ref{Long-time thm}. Therefore, improved estimates on the conservations should be expected.

\begin{corollary} Under the conditions of Theorem \ref{Long-time thm}, when applied to the NLS equation \eqref{NLS eq}, SE1-TSI \eqref{S1} or SE2-TSI \eqref{S2}
has the following near-conservation laws
\begin{align}
 \frac{\eps}{\delta_0^2} \left(H\left(u^n\right)- H\left(u^0\right)\right)&={\mathcal{O}\left(\eps\delta_0^4\right)}+\mathcal{O}(  \delta_{\mathcal{F}}),\quad
 \frac{\eps}{\delta_0^2} \left(I\left(u^n\right)- I\left(u^0\right)\right)={\mathcal{O}\left(\eps \delta_0^4\right)}+\mathcal{O}(  \delta_{\mathcal{F}}),\nonumber\\
 \frac{1}{\delta_0^2} \left(m\left(u^n\right) -m\left(u^0\right)\right)&={\mathcal{O}\left(\eps \delta_0^4\right)}+\mathcal{O}(  \delta_{\mathcal{F}}),\quad 0\leq nh\leq \frac{1}{h}\delta_0^{-N+3} \label{SDE-cor}.
\end{align}
\end{corollary}
\begin{proof}
For the NLS equation,   by letting $u = p + \textmd{i}q$ and $\mathcal{K}=hJ\mathcal{M}$ with  {$\mathcal{M}=\binom{\partial_x^2 \ 0}{\ 0 \ \partial_x^2}$} and $J=\binom{\,0\ -1}{1\ \ 0}$,  it can be transformed into an  infinite-dimensional real-valued Hamiltonian system
\begin{equation*}
\frac{\partial y }{\partial t}=\frac{1}{\eps}J\mathcal{M}y+J\nabla U(y), \ \ \
y_0(x)=  \begin{pmatrix}
                        \textmd{Re}(u_0(x)) \\
                                 \textmd{Im}(u_0(x))
                       \end{pmatrix},
\end{equation*}
where $y=\binom{p}{q}$ and $U(y)=-\frac{1}{4}(p^2+q^2)^2 .$
This is exactly the form \eqref{model} with  $M=\mathcal{M}$. 

The cubic nonlinearity leads to smaller bounds of the coefficients $ \eta^{\mathbf{k}}(t)$:
\begin{equation*}%
\begin{array}{ll}
 \eta_{j,l}^{\langle j\rangle_l}(t)=\mathcal{O}(\delta_0),\quad
  \dot{\eta}_{j,l}^{\langle j\rangle_l}(t)=\mathcal{O}\left(\frac{\eps^2}{h}\nu \delta_0^4\right),
 \quad
 \eta_{j,l}^{\mathbf{k}}(t)=\mathcal{O}\Big(\eps \nu \delta_0^{|\mathbf{k}|+1}\Big),\quad   \mathbf{k}\neq \langle j\rangle_l,
\end{array} %
\end{equation*}
and $ \zeta^{\mathbf{k}}(t)$ have the same estimates.
Therefore, the bounds of the remainders in the almost-invariants are modified accordingly. Let us consider $\mathcal{M}(t)$ as an example here. It becomes
\begin{equation*}\begin{aligned}
\mathcal{M}(t)= \sum\limits_{l=1}^{2d} \sum\limits_{j=-N_{\tau}/2}^{N_{\tau}/2}
  \abs{\zeta^{\langle j \rangle_l}_{j,l}}^2(t)+\mathcal{O}\left(\nu\eps^2\delta_0^5\right)
  +\mathcal{O}\left(\nu^2\eps^2\delta_0^6\right).
\end{aligned}
\end{equation*}
Hence,  the mass at the numerical solution has the following relation with $\mathcal{M}$:
\begin{equation*}
\begin{aligned}
 m(u^n)
  =&  \sum\limits_{l=1}^{2d}
\sum\limits_{j=-N_{\tau}/2}^{N_{\tau}/2} \norm{\zeta^{\langle
{j}\rangle_l}}^2(t_n)+\mathcal{O}\left(  \nu^2 \eps^2 \delta_0^6\right)+\mathcal{O}\left( \delta_{\mathcal{F}}\right)\\
  =&  \sum\limits_{l=1}^{2d}
\sum\limits_{j=-N_{\tau}/2}^{N_{\tau}/2} \abs{\zeta^{\langle
{j}\rangle_l}_{j,l}}^2(t_n)+\mathcal{O}\left(  \nu^2 \eps^2 \delta_0^6\right)+\mathcal{O}(\delta_{\mathcal{F}})\\
=&\mathcal{M}(t_n)+\mathcal{O}\left(\nu\eps^2\delta_0^5\right)+\mathcal{O}\left( \nu^2 \eps^2 \delta_0^6\right)+\mathcal{O}(\delta_{\mathcal{F}})
=\mathcal{M}(t_n)+\mathcal{O}\left(\eps\delta_0^6\right)+\mathcal{O}(\delta_{\mathcal{F}}),
\end{aligned}
\end{equation*}
which yields the third estimate of \eqref{SDE-cor}. Similarly, the other two estimates can be shown.
\end{proof}

\section{Conclusion}\label{sec:6con}
We considered in this work a class of Hamiltonian systems with a scaling parameter $\eps\in(0,1]$.  When $\eps\ll1$, the solution of the model is highly oscillatory in time which triggers unbounded temporal derivatives and energy, and classical numerical methods become inaccurate and inefficient. Such model problem comes from many physical equations in some limit parameter regime or some perturbation problems with a time-compression scaling. We solved the problem by the two-scale formulation approach, and we proposed two new numerical integrators that are symmetric in time. By numerical experiments on a H\'{e}non-Heiles model, a cubic Schr\"odinger equation and a charged-particle system, the proposed methods were shown to have not only second order uniform accuracy for all $\eps$ but also good long-term behaviours. We established the uniform convergence of the methods at a finite time, and established the near-conservation laws of the methods at the discrete level in long times by means of the modulated Fourier expansion.

\appendix

\section*{Acknowledgements}
This work is partially supported by the Natural Science Foundation of Hubei Province No. 2019CFA007 and the NSFC 11901440.

\bibliographystyle{model-num-names}

\end{document}